\documentclass{amsart}
\usepackage[utf8]{inputenc}
\usepackage{amsmath}
\usepackage{mathrsfs}
\usepackage{scrextend}
\usepackage{enumitem}
\usepackage{bbm}
\usepackage{amssymb}
\usepackage{graphicx}
\usepackage[english]{babel}
\usepackage{amsfonts}
\usepackage{amsthm}
\usepackage{tikz-cd}
\usepackage[all]{xy}
\usepackage{mathabx}
\usepackage{setspace}
\usepackage{enumitem}
\usepackage{mathtools}
\usepackage{xcolor}
\usepackage{braket}
\usepackage{yfonts}
\usetikzlibrary{decorations.markings}
\usetikzlibrary{patterns}
\title{On the local Kan structure and differentiation of simplicial manifolds}
\author{}
\date{}
\newcommand{\s}[0]{simplicial\:}
\newcommand{\bu}[0]{\bullet}
\newcommand{\catname}[1]{\mathbf{#1}}
\newcommand{\pr}[0]{\prime}

\newtheorem{thm}{Theorem}[section]
\newtheorem{defi}[thm]{Definition}
\newtheorem{lemm}[thm]{Lemma}
\newtheorem{coro}[thm]{Corollary}
\newtheorem{prop}[thm]{Proposition}

\theoremstyle{remark}
\newtheorem{rema}[thm]{Remark}
\newtheorem{example}[thm]{Example}

\author[F. Dorsch]{Florian Dorsch}
\address{Florian Dorsch, Georg-August-Universität Göttingen,
Institut für Mathematik, Bunsenstr. 3-5, 37073 Göttingen}
\email{florian.dorsch@mathematik.uni-goettingen.de}

\begin{document}
\maketitle
\begin{abstract}
We prove that arbitrary simplicial manifolds satisfy Kan conditions in a suitable local sense. This allows us to expand a technique for diﬀerentiating higher Lie 
groupoids worked out in \cite{ZR} to the setting of general simplicial manifolds. Consequently, we derive a method to differentiate simplicial manifolds into higher Lie algebroids.

\end{abstract}
\tableofcontents

\section{Introduction}

Unlike simplicial sets and simplicial schemes, for which one can readily find many examples which fail to satisfy Kan conditions, it is hard to find natural
 and interesting examples of simplicial manifolds that do not satisfy (at least weak versions of) Kan conditions. This raises the question of whether simplicial 
manifolds always possess some structure near their base space by satisfying "local Kan conditions". Indeed, this holds true. In this article, we provide a precise
answer to this question.\par
This immediately allows us to differentiate simplicial manifolds in a Lie theoretical way. Recall that in ordinary Lie theory, the diﬀerentiation functor sending a
 Lie group to its Lie algebra only depends on the local structure of the Lie group near the identity element. Something similar should be expected for the
 diﬀerentiation of simplicial manifolds due to their local Kan structure.\par 
 Our strategy is to extend a method of differentiating higher Lie groupoids, which was originally sketched in [Se06] and explicitly carried out in \cite{ZR}, to the setting of general simplicial manifolds.\par
The differentiation functor constructed in \cite{ZR}, which sends a higher Lie groupoid to its respective
infinitesimal counterpart, is part of the Lie theoretic correspondence between higher Lie groupoids and
higher Lie algebroids which is the subject of active study within the realm of higher Lie theory. Higher Lie groupoids are simplicial manifolds characterized by internal horn filler conditions called Kan conditions. In contrast, higher Lie algebroids can be viewed as differential positively graded manifolds, also referred to as NQ-manifolds.\par
The project initiated in \cite{ZR} and set to be concluded in a subsequent article aims to show that any higher Lie groupoid $X$ can be differentiated to a higher Lie algebroid whose underlying graded vector bundle is described by the tangent complex of $X$.
The latter is a replacement for the Lie algebroid of a Lie groupoid and serves as a tangent object for higher Lie groupoids. It first appeared in \cite{getzler} and is defined as the chain complex of vector bundles that is obtained when applying a suitable generalization of the Dold-Kan functor to the simplicial vector bundle $TX$ consisting of the tangent bundles $TX_n,\:n\in \mathbb{N}_0$.\par
In \cite{severa}, Severa introduces a method for differentiating higher Lie groupoids, resulting in the construction of presheaves of graded manifolds. Having introduced an infinitesimal object or “fat point" $D_{\bu}$ for higher Lie groupoids, a higher Lie groupoid $X$ is assigned to the presheaf $\text{Hom}(D_{\bu},X_{\bu})$ which is called the tangent functor of $X$. As the central finding in \cite{ZR} it is shown
that the tangent functor $\text{Hom}(D_{\bu},X_{\bu})$ of a higher Lie groupoid $X$ is representable by its tangent complex. Thus one obtains a more explicit version of the construction presented in \cite{severa}.\par
Considering a higher groupoid $X$ as a collection of manifolds $(X_n)_{n\in \mathbb{N}}$, the base space $X_0$ is uniquely embeddable into the higher levels $X_n,\:n \in \mathbb{N}$, thus may be regarded as a submanifold of them. The higher Lie algebroid described by the tangent complex of $X$ now should only depend on the local structure of $X$ near its base space $X_0$, which will be shown inter alia in section 4 of this paper.\par
Notice that for general simplicial manifolds which lack the extra criteria making them higher Lie groupoids, their tangent complex can still be defined point-wise.
Because of the above-mentioned local structure of simplicial manifolds, these point-wise tangent complexes can be patched together into a well-defined  tangent complex of 
vector bundles. \par
We then adapt the argumentation in \cite{ZR} to obtain a differentiation functor for general simplicial manifolds.
In parallel with \cite{ZR}, the primary result of this paper is that for a simplicial manifold $X$ the presheaf $\text{Hom}(D_{\bu},X_{\bu})$ (which can also be defined for general \s manifolds) is representable by the N-manifold whose underlying graded vector bundle is 
described by the tangent complex of $X$. A subsequent article will deal with the computation of the $L_{\infty}$-algebroid structure and higher Lie brackets on the tangent complex.\par
Many naturally occurring objects, such as those derived from singular foliations, often possess inherent simplicial structures. The results of this article provide a framework to derive the infinitesimal structure of these simplicial manifolds directly, bypassing the need to verify the Kan condition, i.e. without confirming whether they qualify as higher groupoids. This approach broadens the scope of differentiation methods applicable to naturally occurring simplicial structures.\par
\vspace{.5cm}
\noindent \textbf{Acknowledgments.}\:\,
I would especially like to thank Leonid Ryvkin and Chenchang Zhu for fruitful discussions and support. I also wish to thank Bertrand Toën,  Miquel Cueca, Ruben Louis, Hao Xu, and all the other members of the Higher Structures seminar in Göttingen as well as the participants of the "BGW Higher Structures PhD retreat 2023" for inspiring conversations, ideas and suggestions. This work was supported by the Research Training Group 2491.

\section{Simplicial objects, (generalized) horns and Kan conditions}
In this first section we revisit some fundamental concepts related to \s objects and
Kan conditions which provides a backdrop for the subsequent discussions of this paper.  For a more thorough exploration of this subject, we direct the interested reader to the works of \cite{duskin}, \cite{may} and \cite{chenchang}.\\\\
Let $\Delta$ be the simplex category which is the category of finite ordinals $[n]:=\{0,\dots,n\}$, $n\in \mathbb{N}_0$, with order-preserving maps as morphisms. Let $\Delta_{\leq n}$ denote the full subcategory on the objects $[0],\dots,[n]$ and $\Delta_{+}$ be the wide subcategory whose morphisms are injective order-preserving functions.
\begin{defi}
   Let $\mathcal{C}$ be a category. A contravariant functor from $\Delta, \Delta_{\leq n},\Delta_{+}$ to $\mathcal{C}$ is called a \textbf{\s, $\mathbf{n}$-\s, semi\s object in $\mathcal{C}$}, respectively. Natural transformations between \s, $n$-\s, semi\s objects are called \s, $n$-\s, semi\s morphisms, respectively. A covariant functor from $\Delta$ to $\mathcal{C}$ is called a \textbf{co\s object} \textbf{in} $\mathbf{\mathcal{C}}$. 
\end{defi}
\noindent Morphisms in the simplex category can be uniquely factorized as compositions of special maps, called cofaces and codegeneracies. Cofaces are the injective order-preserving maps $\delta^n_j:[n-1]\rightarrow [n]$ skipping an element $j$. Codegeneracies are the surjective order-preserving maps $\sigma^n_j:[n+1]\rightarrow [n]$ mapping to an element $j$ twice.\\
Hence, a \s object $X$ in a category $\mathcal{C}$ can be described a collection of objects $(X_n)_{n\in \mathbb{N}_0}$ in $\mathcal{C}$ together with morphisms 
\begin{align*}
    d_i^n:=X(\delta_i^n):X_n=X([n])\rightarrow X_{n-1}=X([n-1]) \;\;\;\;\;\;\text{and}\;\;\;\;\;\;s_i^n:=X(\sigma_i^n):X_n=X([n])\rightarrow X_{n+1}=X([n+1])
\end{align*} 
called faces and degeneracies. Therefore a \s object $X$ is also denoted by $X_{\bu}$ in the following. Faces and degeneracies satisfy the following so-called \textbf{\s identities}:
\begin{gather*}
    d_i^{n}d_j^{n+1}=d_{j-1}^{n}d_i^{n+1}\,\,\,\:\:\:\:\text{if}\,\:\:\:i<j\\
    s_i^ns_j^{n-1}=s_{j+1}^ns_i^{n-1}\,\,\,\,\:\:\:\:\text{if}\,\:\:\:i\leq j\\
    d_i^ns_j^{n-1}=\begin{cases}s_{j-1}^{n-2}d_i^{n-1}\,\,\,\,\,\:\:\:\:\text{if}\,\:\:\:i<j\\
    \text{id}\:\:\:\:\:\:\:\:\:\,\,\,\,\,\,\,\,\,\,\:\:\:\:\:\:\text{if}\,\:\:\:i=j\,\, \text{or}\,\, i+1=j\\
    s_j^{n-2}d_{i-1}^{n-1}\,\,\,\,\,\:\:\:\:\text{if}\,\:\:\:i>j+1.
    \end{cases}
\end{gather*}
Similar considerations as above can be made for $n$-\s and semi\s objects. For simplicity we will omit the upper indices of faces and degeneracies from now on if possible.\\
\\
For a category $\mathcal{C}$, let $s\mathcal{C}$ denote the category of \s objects in $\mathcal{C}$. Let $\text{Yon}:\Delta \rightarrow \catname{sSet}$ be the Yoneda embedding. For $n\in \mathbb{N}_0$, the \s set $\text{Yon}([n]):=\Delta^n$ is called the  standard simplicial $n$-simplex. The image $\text{Im}(\text{Yon}(\delta_j^n):\Delta^{n-1}\rightarrow \Delta^n)=:\partial_j\Delta^n$ is a sub\s set of $\Delta^n$, called the $j$th boundary of $\Delta^n$. For $i\in [n]$, the \s subset $\Delta^n\supseteq\bigcup_{j\in [n]\setminus\{i\}}\partial_j\Delta^n=:\bigwedge^n_i$ is called the \textbf{$\mathbf{(n,i)}$-horn}. More generally, for a subset $I\subseteq [n]$ define $\bigcup_{j\in [n]\setminus I}\partial_j\Delta^n=:\bigwedge^n_I$ to be the $\mathbf{(n,I)}$\textbf{-horn}. Given a \s set $X$, there exist natural bijections
\begin{gather*}
    \text{hom}_{\catname{sSet}}(\bigwedge^n_I,X)\cong \bigwedge^n_I(X):=
    \{(x_{n-1}^j)_{j\in [n]\setminus I}\in X_{n-1}^{\times [n]\setminus I}\,|\,d_{j_1}(x_{n-1}^{j_2})=d_{j_2-1}(x_{n-1}^{j_1})\:\:\text{for}\:\:j_1<j_2;\:j_1,j_2\in [n]\setminus I\}.
\end{gather*}
The set $\bigwedge^n_I(X)$ is called the $\mathbf{(n,I)}$\textbf{-horn space} of $X$ (see also \cite[sec.\,2.1]{duskin}).\\
\begin{figure}
\centering
\begin{tikzpicture}[line join = round, line cap = round]

\coordinate [label=below:2] (2a) at (1,0,0);
\coordinate [label=below:0] (0a) at (0,0,0);
\coordinate [label=above:1] (1a) at (0.5,{sqrt(2)},0);

\begin{scope}[decoration={    markings,    mark=at position 0.5 with {\arrow{to}}}    ]

    \draw[thick,postaction={decorate}] (0a)--(1a);
    \draw[thick,postaction={decorate}] (1a)--(2a);
   
\end{scope}

\node[] at (0.5,-1.5) {$\bigwedge^2_1$};

\coordinate [label=below:2] (2b) at (4.5,0,0);
\coordinate [label=below:0] (0b) at (3.5,0,0);
\coordinate [label=above:1] (1b) at (4,{sqrt(2)},0);

\begin{scope}[decoration={    markings,    mark=at position 0.5 with {\arrow{to}}}    ]

    \draw[thick,postaction={decorate}] (0b)--(1b);
   
\end{scope}

\node[] at (4,-1.5) {$\bigwedge^2_{\{0,1\}}$};

\coordinate [label=above:3] (3c) at (8,{sqrt(2)},0);
\coordinate [label=left:1] (1c) at ({8-.5*sqrt(3)},0,-.5);
\coordinate [label=below:0] (0c) at (8,0,1);
\coordinate [label=right:2] (2c) at ({8+.5*sqrt(3)},0,-.5);

\begin{scope}[decoration={markings,mark=at position 0.5 with {\arrow{to}}}]

\draw[fill=gray,fill opacity=.7] (1c)--(0c)--(3c)--cycle;
\draw[fill=gray,fill opacity=.7] (0c)--(1c)--(2c)--cycle;
\draw[postaction={decorate}] (0c)--(2c);
\draw[postaction={decorate}] (0c)--(1c);
\draw[postaction={decorate}] (1c)--(2c);
\draw[dotted,postaction={decorate}] (2c)--(3c);
\draw[postaction={decorate}] (1c)--(3c);
\draw[postaction={decorate}] (0c)--(3c);
\end{scope}

\node[] at (8.25,-1.5,0) {$\bigwedge^3_{\{0,1\}}$};

\coordinate [label=above:3] (3d) at (12.5,{sqrt(2)},0);
\coordinate [label=left:1] (1d) at ({12.5-.5*sqrt(3)},0,-.5);
\coordinate [label=below:0] (0d) at (12.5,0,1);
\coordinate [label=right:2] (2d) at ({12.5+.5*sqrt(3)},0,-.5);

\begin{scope}[decoration={markings,mark=at position 0.5 with {\arrow{to}}}]

\draw[fill=gray,fill opacity=.7] (1d)--(0d)--(3d)--cycle;
\draw[fill=gray,fill opacity=.7] (2d)--(1d)--(3d)--cycle;
\draw[dotted, postaction={decorate}] (0d)--(2d);
\draw[postaction={decorate}] (0d)--(1d);
\draw[postaction={decorate}] (1d)--(2d);
\draw[postaction={decorate}] (2d)--(3d);
\draw[postaction={decorate}] (1d)--(3d);
\draw[postaction={decorate}] (0d)--(3d);
\end{scope}

\node[] at (12.75,-1.5,0) {$\bigwedge^3_{\{1,3\}}$};
\end{tikzpicture}
\caption{The \s set $\bigwedge^n_I$ is obtained by removing the non-degenerate $n$-face as well as the non-degenerate $(n-1)$-faces indexed by $i\in I$ from $\Delta^n$. }
\end{figure}

\begin{defi}
A \textbf{Grothendieck singleton pretopology} on a category $\mathcal{C}$ consists of a collection of morphisms, called covers, satisfying the following properties:
\begin{itemize}
    \item If $U\rightarrow X$ is a cover and $Y\rightarrow X$ is a morphism in $\mathcal{C}$, the corresponding pullback $Y\times_X U$ exists in $\mathcal{C}$ and the natural morphism $Y\times_XU\rightarrow Y$ is a cover again.
    \item The composition of two covers is a cover.
    \item Isomorphisms are covers.
\end{itemize}
\end{defi}
\begin{defi}
   Let $X$ be a \s object in a category $\mathcal{C}$ which is equipped with a Grothendieck singleton pretopology. For $n\in \mathbb{N},\:i\in \mathbb{N}_0$, $0\leq i\leq n$, $X$ satisfies the \textbf{Kan condition} $\text{Kan}\,(n,i)$ if the Kan projection 
   \begin{gather*}
       p^n_i:=d_0^n\times \dots \times \widehat{d_i^n}\times \dots\times d_n^n:X_n\rightarrow \bigwedge^n_i(X)
   \end{gather*}
  has a well-defined codomain and is a cover. Here, the horn space $\bigwedge^n_i(X):=\lim_{\bigwedge^n_i}X$ is defined as the weighted limit\footnote{For an overview of the concept of weighted limits we refer to \cite{riehl}.} of $X:\Delta^{op}\rightarrow \mathcal{C}$ which is weighted by the \s set $\bigwedge^n_i$.\\
  For $n\in \mathbb{N},\:I\subseteq [n]$, $X$ satisfies the \textbf{generalized Kan condition} $\text{Kan}\,(n,I)$ if the horn space $\bigwedge^n_I(X):=\lim_{\bigwedge^n_I}X$ exists and the horn projection 
  \begin{gather*}
        p^n_I:=\bigsqcap_{j\notin I}d_j^n:X_n\rightarrow \bigwedge^n_I(X)
  \end{gather*}
  is a cover. The \s object $X$ satisfies strong Kan conditions $\text{Kan}\,!(n,i)$, $\text{Kan}\,!(n,I)$ if the respective horn projections are isomorphisms. Similarly, one can define (generalized) Kan conditions for $n$-simplicial and semisimplicial objects in $\mathcal{C}$.
\end{defi}

\begin{defi}
   Let $(\mathcal{C},\mathscr{T})$ be a category with Grothendieck singleton pretopology.
   An \textbf{(internal) $\mathbf{n}$-groupoid object} in $(\mathcal{C},\mathscr{T})$ (for $n\in \mathbb{N}\cup \{\infty\}$) is defined as a \s object $X$ in $\mathcal{C}$ which satisfies the Kan conditions $\text{Kan}\,(m,i)$ for $0\leq i\leq m\leq n$ and $\text{Kan}!\,(m,i)$ for $0\leq i\leq m\geq n+1$.
\end{defi}
\noindent The existence of the horn spaces $\bigwedge^n_i(X)$ is not immediately clear and proven in \cite[Cor.\,2.5]{henri}. Internal $n$-groupoids in $(\catname{Mfd},\mathscr{T}_{\text{surj. subm.}})$ are called \textbf{Lie $\mathbf{n}$-groupoids} whereas internal $\infty$-groupoids in $(\catname{Mfd},\mathscr{T}_{\text{subm.}})$ are \textbf{local Lie $\infty$-groupoids}.\\

\noindent Internal $\mathbf{\infty}$-groupoids in $(\mathcal{C},\mathscr{T})$ are also called \textbf{Kan \s objects} in  $(\mathcal{C},\mathscr{T})$. It is well-known that \s vector spaces are Kan \s objects in $(\catname{Vect},\mathscr{T}_{\text{surj. lin. maps}})$ (see for instance \cite[Thm.\,17.1]{may}). Besides, they satisfy the following generalized Kan conditions.

\begin{lemm}
\label{yone}
Let $V_{\bu}$ be a \s vector space. For $n,m\in \mathbb{N},\:i\in \mathbb{N}_0$, $0\leq i\leq m<n$, the generalized horn projections  $p^n_{\{i,m+1,\dots,n\}}:V_n\rightarrow \bigwedge^n_{\{i,m+1,\dots,n\}}(V)\subseteq V_{n-1}^{\oplus ([m]\setminus\{i\})}$ are surjective.
\end{lemm}
\begin{proof}
See Lemma \ref{proofapp1} in the appendix.
\end{proof}

 \section{Local structure of simplicial manifolds and the tangent complex}
 Using the facts that manifolds are locally approximated by their tangent spaces and that for a simplicial manifold $X_{\bu}$ the simplicial vector spaces $T_{x_0}X_{\bu},\:x_0\in X_0$, are Kan complexes, it will be shown that simplicial manifolds always satisfy Kan conditions on the level of germs at the base space $X_0$. In particular, this makes it possible to find a well-defined notion of tangent complex for \s manifolds.
 \subsection{Local structure of \s manifolds}
\noindent Let $X_{\bu}$ be a \s manifold and regard it as a semi\s manifold by forgetting the degeneracy maps. As proven in the next theorem, one can always find a semi\s submanifold $U_{\bu}$ of $X_{\bu}$ consisting of open neighborhoods of the base space $X_0$ such that $U_{\bu}$ satisfies the undermentioned (generalized) Kan conditions.
\begin{thm}
\label{yummiiiiiiiiii}
Let $X_{\bu}$ be a \s manifold. There exist open neighborhoods $U_n$ of $X_0\subseteq X_n$, $n\in \mathbb{N}_0$, such that the face maps of $X_{\bu}$ can be restricted to well-defined maps between the $U_n$’s and the semi\s manifold $U_{\bu}$ satisfies the following (generalized) Kan condition for $n\in \mathbb{N}$:
\begin{enumerate}
 \renewcommand{\labelenumi}{\roman{enumi})}
    \item $p^n_i:U_n\rightarrow \bigwedge^n_i(U)$ is a submersion for $0\leq i\leq n$, where $\bigwedge^n_i(U)$ is a well-defined embedded submanifold of $U_{n-1}^{\times n}$
    \item $p^n_{\{i,m+1,\dots,n\}}:U_n\rightarrow \bigwedge^n_{\{i,m+1,\dots,n\}}(U)$ is a submersion for $1\leq m<n,\:0\leq i\leq m$, where $\bigwedge^n_{\{i,m+1,\dots,n\}}(U)$ is a well-defined embedded submanifold of $U_{n-1}^{\times m}$.
\end{enumerate}
\end{thm}
\begin{proof}
We first note that for a \s manifold $X_{\bu}$ its degeneracy maps $s_i:X_n\rightarrow X_{n+1}$ are embeddings since they have smooth left inverses. In particular, the base space $X_0$ can be embedded (uniquely) via iterated degeneracy maps into the manifolds $X_n,\:n\geq 1$. When identifying a given point $x_0\in X_0$ with $(s_0)^n(x_0),\:n\in \mathbb{N}$, the tangent spaces $(T_{x_0}X_n)_{n\in \mathbb{N}}$ form a \s vector space $T_{x_0}X_{\bu}$.\\
We now prove the statement inductively.
For $n=1$ set $U_0:=X_0$ and $\bigwedge^1_0(U)=\bigwedge^1_1(U):=X_0$. Because of the \s identity $d_0^1s_0^0=d_1^1s_0^0=id_{X_0}$ one can find an open neighborhood $U_1$ of $X_0\subseteq X_1$ such that ${d_0}_{|U_1},{d_1}_{|U_1}$ are submersions. Then the 1-Kan conditions (condition i)) are satisfied and condition ii) is empty. Define $\bigwedge^2_{\{0,2\}}(U)=\bigwedge^2_{\{1,2\}}(U):=U_1$ and the 2-horn spaces $\bigwedge^2_i(U),\:0\leq i\leq 2$, as ordinary fiber products.\\
Now suppose that the statement is true until $n$ (and that the corresponding horn spaces exist). To show that i) and ii) hold for $n+1$, first define the generalized horn spaces $\bigwedge^{n+1}_{\{i,m+1,\dots,n+1\}}(U),  \:0\leq i\leq m\leq n$. Let us recall that $\bigwedge^{n+1}_{\{i,m+1,\dots,n+1\}}(U)$ is the set of $m$-tuples $(x_n^0,\dots,\widehat{x_n^i},\dots,x_n^m)\in U_n^{\times m}$ satisfying $d_{i_1}(x_n^{i_2})=d_{i_2-1}(x_n^{i_1})$ for $i_1<i_2;\:i_1,i_2\in \{0,\dots,m\}\setminus\{i\}$. For $m\geq 2,\:0\leq i\leq m$, an $m$-tuple $(x_n^0,\dots,\widehat{x_n^i},\dots,x_n^m)\in U_n^{\times m}$ lies in $\bigwedge^{n+1}_{\{i,m+1,\dots,n+1\}}(U)$ if and only if
\begin{gather*}
    (x_n^0,\dots,\widehat{x_n^i},\dots,x_n^{m-1})\in \bigwedge^{n+1}_{\{i,m,\dots,n+1\}}(U)\:\:\:\text{and}\:\:\:d_{m-1}\times \dots\times d_{m-1}((x_n^0,\dots,\widehat{x_n^i},\dots,x_n^{m-1}))=p^n_{\{i,m,\dots,n\}}(x_n^m)\:\:\text{if}\:\:\:i<m\\
    (x_n^0,\dots,x_n^{m-2})\in \bigwedge^{n+1}_{\{m-1,m,\dots,n+1\}}(U)\:\:\:\text{and}\:\:\:d_{m-2}\times \dots\times d_{m-2}((x_n^0,\dots,x_n^{m-2}))=p^n_{\{m-1,m,\dots,n\}}(x_n^{m-1})\:\:\:\:\:\text{if}\:\:\:i=m
\end{gather*}
which allows to define the generalized horn spaces recursively.
\noindent Set $\bigwedge^{n+1}_{\{0,2,\dots,n+1\}}(U)=\bigwedge^{n+1}_{\{1,2,\dots,n+1\}}(U):=U_n$. Supposing that the horn spaces $\bigwedge^{n+1}_{\{i,r+1,\dots,n+1\}}(U)$ exist for $1\leq r\leq m-1,\:0\leq i\leq r$, recursively define\footnote{As illustrated in Figure 2, the horn spaces are constructed by gradually adding faces with increasing indices.\label{bonobo}} the horn spaces
\begin{gather*}
    \bigwedge^{n+1}_{\{i,m+1,\dots,n+1\}}(U):=\begin{cases}\bigwedge^{n+1}_{\{i,m,\dots,n+1\}}(U)\times_{d_{m-1}\times \dots \times d_{m-1},\bigwedge^n_{\{i,m,\dots,n\}}(U),p^n_{\{i,m,\dots,n\}}}U_n\:\:\:\:\:\:\:\:\:\:\:\:\:\:\:\:\:\:\:\:\:\:\text{if}\:\:\:\:i<m\\
    \bigwedge^{n+1}_{\{m-1,m,\dots,n+1\}}(U)\times_{d_{m-2}\times \dots \times d_{m-2},\bigwedge^n_{\{m-1,m,\dots,n\}}(U),p^n_{\{m-1,m,\dots,n\}}}U_n\:\:\:\:\:\text{if}\:\:\:\:i=m
\end{cases}
\end{gather*}
which are well-defined fiber products because the horn projections $p^n_{\{i,m,\dots,n\}}:U_n\rightarrow \bigwedge^n_{\{i,m,\dots,n\}}(U)$, $0\leq i\leq m-1$, are submersions by induction hypothesis. 
Using the fact that the horn projections $p^n_i:U_n\rightarrow \bigwedge^n_i(U),\:0\leq i\leq n$, are submersions, we now similarly define\footref{bonobo}
\begin{gather*}
    \bigwedge^{n+1}_i(U):=\begin{cases}\bigwedge^{n+1}_{\{i,n+1\}}(U)\times_{d_n\times \dots\times d_n,\bigwedge^n_i(U),p^n_i}U_n\:\:\:\:\:\:\:\:\:\,\:\:\:\:\:{\text{if}}\:\,\, i<n+1,\\
  \bigwedge^{n+1}_{\{n,n+1\}}\times_{d_{n-1}\times \dots\times d_{n-1},\bigwedge^n_n(U),p^n_n}U_n\:\:\:\:\:\:\:\:\:\:\:\:\:\:{\text{if}}\:\:\: i=n+1.
    \end{cases}
\end{gather*}
\begin{figure}
\centering
\begin{tikzpicture}[line join = round, line cap = round]

\coordinate [label=above:3] (3a) at (0,{sqrt(2)},0);
\coordinate [label=left:0] (0a) at ({-.5*sqrt(3)},0,-.5);
\coordinate [label=below:1] (1a) at (0,0,1);
\coordinate [label=right:2] (2a) at ({.5*sqrt(3)},0,-.5);

\begin{scope}[decoration={markings,mark=at position 0.5 with {\arrow{to}}}]

\draw[pattern=vertical lines, pattern color=black] (2a)--(0a)--(3a)--cycle;
\draw[postaction={decorate}] (0a)--(2a);
\draw[dotted, postaction={decorate}] (0a)--(1a);
\draw[dotted, postaction={decorate}] (1a)--(2a);
\draw[postaction={decorate}] (2a)--(3a);
\draw[dotted, postaction={decorate}] (1a)--(3a);
\draw[postaction={decorate}] (0a)--(3a);
\end{scope}

\node[] at (0.25,-1.5,0) {$\bigwedge^3_{\{0,2,3\}}(U)$};

\coordinate [label=above:3] (3b) at (4,{sqrt(2)},0);
\coordinate [label=left:0] (0b) at ({4-.5*sqrt(3)},0,-.5);
\coordinate [label=below:1] (1b) at (4,0,1);
\coordinate [label=right:2] (2b) at ({4+.5*sqrt(3)},0,-.5);

\begin{scope}[decoration={markings,mark=at position 0.5 with {\arrow{to}}}]

\draw[pattern= north west lines, pattern color=black] (1b)--(0b)--(3b)--cycle;
\draw[pattern=vertical lines, pattern color=black] (2b)--(0b)--(3b)--cycle;
\draw[postaction={decorate}] (0b)--(2b);
\draw[postaction={decorate}] (0b)--(1b);
\draw[dotted, postaction={decorate}] (1b)--(2b);
\draw[postaction={decorate}] (2b)--(3b);
\draw[postaction={decorate}] (1b)--(3b);
\draw[postaction={decorate}] (0b)--(3b);
\end{scope}

\node[] at (4.25,-1.5,0) {$\bigwedge^3_{\{0,3\}}(U)$};

\coordinate [label=above:3] (3c) at (8,{sqrt(2)},0);
\coordinate [label=left:0] (0c) at ({8-.5*sqrt(3)},0,-.5);
\coordinate [label=below:1] (1c) at (8,0,1);
\coordinate [label=right:2] (2c) at ({8+.5*sqrt(3)},0,-.5);

\begin{scope}[decoration={markings,mark=at position 0.5 with {\arrow{to}}}]

\draw[pattern= north west lines, pattern color=black] (1c)--(0c)--(3c)--cycle;
\draw[pattern=horizontal lines, pattern color=black] (1c)--(0c)--(2c)--cycle;
\draw[pattern=vertical lines, pattern color=black] (2c)--(0c)--(3c)--cycle;
\draw[postaction={decorate}] (0c)--(2c);
\draw[postaction={decorate}] (0c)--(1c);
\draw[postaction={decorate}] (1c)--(2c);
\draw[postaction={decorate}] (2c)--(3c);
\draw[postaction={decorate}] (1c)--(3c);
\draw[postaction={decorate}] (0c)--(3c);
\end{scope}

\node[] at (8.25,-1.5,0) {$\bigwedge^3_{0}(U)$};

\node[] at (2.2,0.5,0)   {$\leadsto$};
\node[] at (6.2,0.5,0)   {$\leadsto$};
\end{tikzpicture}
 \caption{The construction of $\bigwedge^3_{0}(U)$ }
\end{figure}
\noindent Set $\Tilde{U}_{n+1}:=\bigcap_{j=0}^{n+1}d_j^{-1}(U_n)$ which is an open neighborhood of $X_0\subseteq X_{n+1}$. Then one can define well-defined horn projections \begin{gather*}
    p^{n+1}_i:\tilde{U}_{n+1}\rightarrow \bigwedge^{n+1}_i(U),\:\:\: 0\leq i\leq n+1,\\
    p^{n+1}_{\{i,m+1,\dots,n+1\}}:\tilde{U}_{n+1}\rightarrow \bigwedge^{n+1}_{\{i,m+1,\dots,n+1\}}(U),\:\:\:0\leq i\leq m\leq n.
\end{gather*}
For $x_0\in X_0$ the tangent maps
\begin{gather*}
    {Tp^{n+1}_i}_{|x_0}:T_{x_0}\Tilde{U}_{n+1}=T_{x_0}X_{n+1}\rightarrow T_{(x_0,\dots,x_0)}\bigwedge^{n+1}_i(U)=\bigwedge^{n+1}_i(T_{x_0}X),\\
    {Tp^{n+1}_{\{i,m+1,\dots,n+1\}}}_{|x_0}:T_{x_0}\Tilde{U}_{n+1}=T_{x_0}X_{n+1}\rightarrow T_{(x_0,\dots,x_0)}\bigwedge^{n+1}_{\{i,m+1,\dots,n+1\}}(U)=\bigwedge^{n+1}_{\{i,m+1,\dots,n+1\}}(T_{x_0}X)
\end{gather*}
are surjective which follows from the fact that \s vector spaces are Kan and Lemma \ref{yone}.
So there exists an open neighborhood $U_{n+1}$ of $X_0\subseteq X_{n+1}$ such that $U_{n+1}$ is a subset of $\Tilde{U}_{n+1}$ and the restrictions ${p^{n+1}_i}_{|U_{n+1}}:U_{n+1}\rightarrow \bigwedge^{n+1}_i(U),\:0\leq i\leq n+1$, $ {p^{n+1}_{\{i,m+1,\dots,n+1\}}}_{|U_{n+1}}:U_{n+1}\rightarrow \bigwedge^{n+1}_{\{i,m+1,\dots,n+1\}}(U),\:0\leq i\leq m\leq n$, are submersions. 
Thus, i) and ii) hold for $n+1$ and the statement follows by induction.
\end{proof}

\begin{coro}
\label{kaisa}
Let $f_{\bu}:X_{\bu}\rightarrow Y_{\bu}$ be a \s morphism between \s mani\-folds. Then there exist Kan semi\s manifolds $U_{\bu},V_{\bu}$ induced by $X_{\bu},Y_{\bu}$ (and depending on $f_{\bu}$) such that $f_{\bu}$ can be restricted to a well-defined semi\s morphism $f_{\bu}:U_{\bu}\rightarrow V_{\bu}$.
\end{coro}
\begin{proof}
Construct $V_{\bu}$ like in the proof of Theorem\,\ref{yummiiiiiiiiii} and set $U_0:=X_0$. When recursively constructing the sets $U_n$ for $n\geq 1$ shrink $U_n$ if necessary, so it is a subset of $f_n^{-1}(V_n)$.
\end{proof}

\noindent Motivated by constructions made in the proof of Theorem \ref{yummiiiiiiiiii}, we introduce the following definition.

 \begin{defi}
    A \textbf{Kan semi\s manifold} $X_{\bu}$ is a semi\s manifold which satisfies the following generalized Kan conditions for $n\in \mathbb{N}$:
    \begin{gather*}
      i)\:\:p^n_i:X_n\rightarrow \bigwedge^n_i(X)\:\:\:\text{is a submersion for}\:\: 0\leq i\leq n,\\  
      ii)\:\:p^n_{\{i,m+1,\dots,n\}}:X_n\rightarrow \bigwedge^n_{\{i,m+1,\dots,n\}}(X)\:\:\:\text{is a submersion for}\:\: 0\leq i\leq m,\:1\leq m\leq n-1.
    \end{gather*}
    \end{defi}

\begin{rema}
  Unlike for Kan simplicial manifolds, the requirement of the “usual” Kan conditions (i) does not guarantee the existence of the horn spaces $\bigwedge^n_i(X)$ of a semi\s manifold $X_{\bu}$. If one additionally assumes that the generalized Kan conditions (ii) hold (which turn out to be automatically satisfied in the case of Kan simplicial manifolds (see \cite[Lem.\,2.4]{henri})), the aforementioned horn spaces $\bigwedge^n_i(X),\:\bigwedge^n_{\{i,m+1,\dots,n\}}(X)$ can be defined recursively as in the previous theorem.  
\end{rema}
\noindent We further introduce several notions which will prove to be useful for our purposes.
\begin{itemize}
    \item A \textbf{semi\s manifold} $U_{\bu}$ is \textbf{induced} by a \s manifold $X_{\bu}$ if its face maps are restrictions of the face maps of $X_{\bu}$ and for each $n\in \mathbb{N},$ $U_n$ is an open neighborhood of $X_0\subseteq X_n$.
    \item Similarly, an \textbf{$\mathbf{n}$-\s manifold} $V_{\bu}$ is \textbf{induced} by a \s manifold $X_{\bu}$ if its structure maps are restrictions of the structure maps of $X_{\bu}$ and $V_m$ is an open neighborhood of $X_0\subseteq X_m$ for $m\in \{0,\dots,n\}$.
    \item A \textbf{Kan $\mathbf{n}$-\s manifold} is an $n$-\s manifold satisfying Kan condition $\text{Kan}(m,i)$ in $(\catname{Mfd},\mathscr{T}_{\text{subm.}})$ for $1\leq m\leq n,\:0\leq i\leq m$.
\end{itemize}
    \noindent Let $X_{\bu},Y_{\bu}$ be $n$-\s manifolds and $Y_{\bu}$ be a Kan $n$-\s manifold. If $X_m$ is an open submanifold of $Y_m$ for $m\in \{0,\dots,n\}$ and the structure maps of $X_{\bu}$ are restrictions of the structure maps of $Y_{\bu}$, then $X_{\bu}$ is a Kan $n$-\s manifold, as well.\\
\noindent It turns out that every induced $n$-semisimplicial manifold can be shrunk to an induced  $n$-simplicial manifold.
\begin{lemm}
\label{blado}
    Let $X_{\bu}$ be a \s manifold and $U_{\bu}$ be an $n$-semi\s manifold  induced by $X_{\bu}$. Then there exists an $n$-\s manifold $U^{\prime}_{\bu}$ induced by $X_{\bu}$ such that $U_m^{\prime}\subseteq U_m$ for $m\in \{0,\dots,n\}$.
\end{lemm}
\begin{proof}
For a multi-index $I=(i_1,\dots,i_l)$, let $s_I:=s_{i_1}\circ \dots \circ s_{i_l}$ denote the corresponding $l$-fold composition of degeneracy maps. Let $|I|$ denote the length of a multi-index $I$.\\
Given a \s manifold $X_{\bu}$ with an induced $n$-semi\s manifold $U_{\bu}$, define the following subsets of the $U_m,\:m\in \{0,\dots,n\}$. Set
\begin{gather*}
    U^{\prime}_0:=U_0=X_0,\:U^{\pr}_1:=U_1\cap\bigcap_{I, 1\leq |I|\leq n-1} s_I^{-1}(U_{1+|I|})
\end{gather*}
and for $m\in \{2,\dots,n-1\}$ define $U^{\pr}_m$ recursively as
\begin{gather*}
    U^{\pr}_m:=U_m\cap \bigcap_{j=0}^m d_j^{-1}(U^{\pr}_{m-1})\cap \bigcap_{I, 1\leq |I|\leq n-m} s_I^{-1}(U_{m+|I|}).
\end{gather*}
For $l\in \{1,\dots,n-1\}$ the intersection $\bigcap_{I,1\leq |I|\leq n-l} s_I^{-1}(U_{l+|I|})$ is defined over the set of all multi-indices of length $1\leq |I|\leq n-l$ such that the iterated degeneracy map $s_I:X_l\rightarrow X_{l+|I|}$ is defined.
For $m=n$ define $U^{\pr}_m$ as
\begin{gather*}
    U^{\pr}_m:=U_m\cap \bigcap_{j=0}^m d_j^{-1}(U^{\pr}_{m-1}).
\end{gather*}
Note that the $U^{\pr}_m$’s are open subsets of the $U_m$’s and contain $X_0$ as it is preserved by the face and degeneracy maps. By construction, the face maps can be restricted to well-defined maps between the $U^{\pr}_m$’s. It remains to be shown that the same is true for the degeneracy maps. We show this inductively. Obviously, $s^0_0$ maps from $U^{\pr}_0$ to $U^{\pr}_1$. Suppose that the degeneracy maps $s^r_{\bu}$ are well-defined for $r\leq m-1$.\\
Now consider a degeneracy map $s^m_i,\:i\in\{0,\dots,m\}$. By construction, $s_{(i,J)}$ maps $U^{\pr}_m$ to $U_{m+1+|J|}$ for every suitable multi-index $J$ of length $1\leq |J|\leq n-m-1$, so $s_i^m$ maps $U^{\pr}_m$ to $U_{m+1}\cap \bigcap_{J, 1\leq |J|\leq n-m-1} s_J^{-1}(U_{m+1+|J|})$. Thus it is sufficient to show that $s_i^m$ maps $U_m^{\prime}$ to $U_{m+1}\cap \bigcap_{j=0}^{m+1} d_j^{-1}(U^{\pr}_{m})$, i.e. that $d_js_i$ maps from $U^{\pr}_m$ to $U^{\pr}_m$ for $j\in \{0,\dots,m+1\}$. For an element $x_m\in U^{\pr}_m$, one has $d_js_i(x_m)\in U_m$, so due to the way $U^{\prime}_m$ is defined one needs to check that 
\begin{gather*}
    s_Id_js_i(x_m)\in U_{m+|I|}\:\:\:\:\text{for every multi-index $I$ with}\:\:|I|\in \{1,\dots, n-m\}\:\:\,\text{and that}\\
    d_ld_js_i(x_m)\in U^{\pr}_{m-1}\:\:\:\:\text{for}\:\:\l\in \{0,\dots,m\}.
\end{gather*}

\noindent Now, using \s identities, one has the following cases:
\begin{gather*}
   s_Id_js_i(x_m)=\begin{cases} s_{\tilde{I}}(x_m)\:\:\:\:\:\:\text{for a multi-index $\Tilde{I}$ with $|\Tilde{I}|=|I|$}\:\:\:\:\:\:\:\:\:\:\:\:\:\:\:\:\:\:\:(1)
   \\
   s_{\tilde{I}}d_{\tilde{j}}(x_m)\:\:\text{for a multi-index $\Tilde{I}$ with $|\Tilde{I}|=|I|+1$}\\
   \text{and an index $\tilde{j}\in \{0,\dots,m\}$}\:\:\:\:\:\:\:\:\:\:\:\:\:\:\:\:\:\:\:\:\:\:\:\:\:\:\:\:\:\:\:\:\:\:\:\:\:\:\:\:\:\:\:\:\:\:\:\:\,(2)
   \end{cases}
\end{gather*}

   \begin{gather*}
       d_ld_js_i(x_m)=\begin{cases}
       s_{\tilde{i}}d_{\tilde{l}}d_{\tilde{j}}(x_m)\:\:\:\:\text{for some indices $\tilde{i},\tilde{l},\tilde{j}$}\:\:\:\:\:\:\:\:\:\:\:\:\:\:\:\:\:\:\:(a)
       \\
       d_{\tilde{j}}(x_m)\:\:\:\:\text{for some index $\tilde{j}\in \{0,\dots,m\}$}\:\:\:\:\:\:\:\:\:\:\:(b)
       \end{cases}.
   \end{gather*}
   Regarding each case, it follows\\
   1)\:that $s_{\tilde{I}}(x_m)\in U_{m+|I|}$,\\
   2)\:that $d_{\tilde{j}}(x_m)\in U^{\pr}_{m-1}$ and consequently $s_{\tilde{I}}d_{\tilde{j}}(x_m)\in U_{m-1+|I|+1}=U_{m+|I|}$,\\
   a)\:that $d_{\tilde{l}}d_{\tilde{j}}(x_m)\in U^{\pr}_{m-2}$ and by induction hypothesis $s_{\tilde{i}}d_{\tilde{l}}d_{\tilde{j}}(x_m)\in U^{\pr}_{m-1}$,\\
   b)\:that $d_{\tilde{j}}(x_m)\in U^{\pr}_{m-1}$.
   \\
   \\
Thus, $d_js_i(x_m)\in U^{\pr}_m$ and $s_i^m:U^{\pr}_m\rightarrow U^{\pr}_{m+1}$ is well-defined.
\end{proof}
\noindent Combining Theorem \ref{yummiiiiiiiiii} and Lemma \ref{blado}, we conclude:
\begin{coro}
\label{blubiblupp}
 Let $X_{\bu}$ be a simplicial manifold. For $n\in \mathbb{N}$ there exists a Kan $n$-simplicial manifold induced by $X_{\bu}$.
\end{coro}
\noindent The preceding corollary can be elegantly reformulated with the introduction of the concept of germs of \s manifolds. We first give the following definition.\footnote{For a short overview of the concept of germ space see also \cite{hu}.}

\begin{defi}
   For a fixed smooth manifold $S$, we define the \textbf{category of $S$-embeddings}, denoted by $\catname{Mfd_S}$, as follows. Objects are pairs $(M,i)$ where $M$ is a smooth manifold and $i$ denotes an embedding $i:S\hookrightarrow M$. Morphisms $f:(M,i)\rightarrow(N,j)$ are $S$-preserving smooth maps, meaning that the composition \mbox{$f\circ i$ equals $j$}.\\
   The \textbf{category $\mathbf{\mathscr{G}_S}$ of germs at $\mathbf{S}$} is defined as the localization of $\catname{Mfd_S}$ at the class of those morphisms that can be restricted to diffeomorphisms on open neighborhoods of $S$. Explicitly, morphisms $[f]:(M,i)\rightarrow (N,j)$ are defined as equivalence classes of $S$-preserving maps $U\rightarrow N$ defined on open neighborhoods $U$ of $S$ in $M$ where two such morphism are equivalent if they agree on an open neighborhood of $S$ contained in the intersection of their domains.
\end{defi}
\noindent Given a \s manifold $X_{\bu}$, its base space $X_0$ is uniquely embeddable into the higher $X_n$’s and preserved by the face and degeneracy maps. Thus, a \s manifold $X_{\bu}$ gives rise to the \s object 
\begin{gather*}
    X_0\substack{
[s_0]\\
\longrightarrow\\
\longleftarrow\\
[d_0],[d_1]}X_1    \substack{
[s_0],[s_1]\\
\longrightarrow\\
\longleftarrow\\
[d_0],[d_1],[d_2]}X_2      \substack{
\longrightarrow\\
\longleftarrow}\cdots
\end{gather*}
in $\mathbf{\mathscr{G}}_{X_0}$ which will be denoted by $g(X)$.

\noindent By declaring morphisms which are represented by submersions defined on open neighborhoods of $X_0$ as covers, one obtains a Grothendieck singleton pretopology on $\mathbf{\mathscr{G}}_{X_0}$, and the local structure of a \s manifold $X_{\bu}$ can be described in the following way.
\begin{thm}
  Let $X_{\bu}$ be a \s manifold. Then $g(X)$ is a Kan \s object in $\mathbf{\mathscr{G}}_{X_0}$.
\end{thm}
\begin{proof}
By Corollary \ref{blubiblupp}, for $n\in \mathbb{N}$ there exists a Kan $n$-\s manifold $U^{(n)}_{\bu}$ induced by $X_{\bu}$. It now can be easily verified that the horn spaces $\bigwedge^m_i(U^{(n)}),\:0\leq i\leq m,\:1\leq m\leq n$, are also the $(m,i)$-horn spaces of $g(X)$ in $\mathbf{\mathscr{G}}_{X_0}$ and that the corresponding horn projections of $g(X)$ are represented by the horn projections $\prescript{(n)}{}{p^m_i}:U^{(n)}_m\rightarrow \bigwedge^m_i(U^{(n)})$ of $U^{(n)}_{\bu}$. Hence, $g(X)$ satisfies the Kan conditions $\text{Kan}(m,i)$ for $0\leq i\leq m,\:1\leq m\leq n$. Since $n\in \mathbb{N}$ was arbitrary, $g(X)$ is a Kan \s object in $\mathbf{\mathscr{G}}_{X_0}$.
\end{proof}

\subsection{The tangent complex}
\noindent Let us quickly recall that the normalized chain complex functor $N:\catname{sVect}\rightarrow \catname{Ch_{\geq 0}(Vect)}$ is the functor sending a \s vector space $V_{\bu}$ to the connected chain complex of vector spaces $(N(V)_{\bu},\partial_{\bu})$ with terms 
\begin{gather*}
    N(V)_n=\begin{cases}V_0\:\:\:\:\:\:\:\:\:\:\:\:\:\:\:\:\:\:\:\:\:\:\:\:\:\:\:\:\:\:\:\:\:\:\:\:\:\:\:\:\:\:\:\:\:\:\:\:\:\:\:\:\:\:\:\:\:\:\:\:\:\:\:\:\:\:\:\:\:\:\:\:n=0\\
    \bigcap_{j=0}^{n-1}\text{ker}\,(d_j:V_n\rightarrow V_{n-1})\:\:\:\:\:\:\:\:\:\:\:\:\:\:\:\:\:\:\:\:\:\:\:\:\:n\geq 1
    \end{cases}
\end{gather*}
and differentials given by $\partial_n=(-1)^nd_n^n,\:n\in \mathbb{N}$.\\
The \textbf{tangent complex} of a \s manifold $X_{\bu}$ is obtained when applying the normalization functor to the \s vector bundle $TX_{\bu}$, i.e. it is the chain complex of vector bundles $(\mathfrak{T}_{\bu}(X),\partial_{\bu})$ whose objects are defined as 
    \begin{gather*}
        \mathfrak{T}_n(X):=\bigsqcup_{x_0\in X_0}N_n(T_{x_0}X_{\bu})=
        \begin{cases}
        TX_0,\:\:\:\:\:\:\:\:\:\:\:\:\:\:\:\:\:\:\:\:\:\:\:\:\:\:\:\:\:\:\:\:\:\:\:\:\:\:\:\:\:\:\:\:\:\:\:\:\:\:\:\:\:\:\:\:\:\:\:\:\:\:\:\:\:\:\:n=0\\
        \bigcap_{j=0}^{n-1}\text{ker}\,(Td_j:{TX_n}_{|X_0}\rightarrow {TX_{n-1}}_{|X_0}),\:n\geq 1
        \end{cases}
    \end{gather*}
and whose differentials are given by $\partial_n=(-1)^nTd_n^n,\:n\in \mathbb{N}$.\\
Note that it is not obvious why the sets $\mathfrak{T}_n(X)$ possess a smooth structure making them vector subbundles of ${TX_n}_{|X_0}$. In the case of local Lie $\infty$-groupoids, $\mathfrak{T}_n(X)=\text{ker}\,(Tp^n_n:TX_n\rightarrow T\bigwedge^n_n(X))_{|X_0}$ is a well-defined subbundle of ${TX_n}_{|X_0}$ because $p^n_n$ is a submersion. Since the tangent complex “only sees the local structure near the base space" and \s manifolds satisfy local Kan conditions, it still proves to be well-defined in the general setting, and further can be realized as a functor.
\begin{prop}
Assigning the tangent complex is a functor $\mathfrak{T}:\catname{sMfd}\rightarrow \catname{VBChaincompl.}$
\end{prop}
\begin{proof}
For a \s manifold $X_{\bu}$ with an induced Kan semi\s manifold $U_{\bu}$ define $\mathfrak{T}_n(X):=\text{ker}\,({Tp^n_n}:{TU_n}\rightarrow T\bigwedge^n_n(U))_{|X_0}$ which is a subbundle of ${TU_n}_{|X_0}$. The vector bundle morphisms $(-1)^nTd_n:TU_n\rightarrow TU_{n-1}$ can be restricted to well-defined smooth maps $\partial_n=(-1)^n{Td_n}_{|X_0}:\mathfrak{T}_n(X)\rightarrow \mathfrak{T}_{n-1}(X)$. Thus $(\mathfrak{T}_{\bu}(X),\partial_{\bu})$ is a chain complex of vector bundles.
\\
\noindent Given a \s morphism $f_{\bu}:X_{\bu}\rightarrow Y_{\bu}$ between \s manifolds, by Corollary \ref{kaisa} there exist induced Kan semi\s mani\-folds $U_{\bu}, V_{\bu}$ so that $f_{\bu}$ can be restricted to a well-defined semi\s morphism $f_{\bu}:U_{\bu}\rightarrow V_{\bu}$. As $Tf_n$ maps $\mathfrak{T}_n(X)$ to $\mathfrak{T}_n(Y)$ and $\mathfrak{T}_n(X)$, $\mathfrak{T}_n(Y)$ are subbundles of ${TU_n}_{|X_0},{TV_n}_{|Y_0}$, the map ${Tf_n}_{|X_0}:{TU_n}_{|X_0}\rightarrow {TV_n}_{|Y_0}$ can be restricted to $\mathfrak{T}_n(f):=Tf_n:\mathfrak{T}_n(X)\rightarrow \mathfrak{T}_n(Y)$. Note that the definition of the functor $\mathfrak{T}$ does not depend on the choices made.
\end{proof}

\begin{rema}
\label{twoversions}
Equivalently, the tangent complex of a \s manifold $X_{\bu}$ may also be defined as the complex of vector bundles $(\tilde{\mathfrak{T}_{\bu}}(X),\tilde{\partial}_{\bu})$ whose objects are 
\begin{gather*}
    \tilde{\mathfrak{T}}_n(X):=\begin{cases}TX_0\:\:\:\:\:\:\:\:\:\:\:\:\:\:\:\:\:\:\:\:\:\:\:\:\:\:\:\:\:\:\:\:\:\:\:\:\:\:\:\:\:\:\:\:\:\:\:\:\:\:\:\:\:\:\:\:\:\,\:\:\:\:\:\:\:\:\:\:\:\:\:n=0\\
    \bigcap_{j=1}^{n}\text{ker}\,(Td_j:{TX_n}_{|X_0}\rightarrow {TX_{n-1}}_{|X_0}),\:\:\:n\geq 1
    \end{cases}
\end{gather*}
and whose differentials are given by $\tilde{\partial}_n=Td_0^n$. Analogously like before, $\tilde{\mathfrak{T}}$ proves to be a well-defined functor valued in the category of VB-chain complexes. Using Lemma \ref{blado}, the isomorphism presented in Lemma \ref{useless} can be easily extended to a natural isomorphism between the two versions of the tangent complex. The second definition will be used in the following.

\end{rema}

\section{Differentiating \s manifolds}
\noindent Following the presentation of essential preliminary definitions and the introduction of the tangent functor, we will prove that the tangent functor of a \s manifold is representable by its tangent complex.\\
\noindent For the sake of better readability we recall several results and definitions from \cite{ZR} which are most relevant for the purpose of this paper in appendix B.

\subsection{Preliminary definitions}
\subsubsection{$\mathbb{Z}$-\,and $\mathbb{N}$-graded manifolds}
In the following we give some basic definitions about $\mathbb{Z}$-graded and $\mathbb{N}$-graded manifolds. Regarding $\mathbb{Z}$-graded manifolds we adopt the conventions of \cite{ZR} and refer the reader to \cite{vysoki} and \cite{salni} for a comprehensive introduction on $\mathbb{Z}$-graded manifold theory.
\begin{defi}
    A \textbf{$\mathbf{\mathbb{Z}}$-graded manifold} is a pair $\mathcal{M}=(M,C(\mathcal{M}))$ where $M$ is a smooth manifold, called the \textbf{body}, and $C(\mathcal{M})$ is a $\mathbb{Z}$-graded commutative unital algebra, called the \textbf{function algebra} of $\mathcal{M}$. Further, there is an isomorphism $C(\mathcal{M})\cong \Gamma S^{\bu}(E_{\bu}^{\ast})$ where $E_{\bu}^{\ast}$ is the dual of a $\mathbb{Z}$-graded vector bundle $E_{\bu}$ and $S^{\bu}$ denotes the graded symmetric product.\\
    A morphism between $\mathbb{Z}$-graded manifolds $\mathcal{M,N}$ is defined as a $\mathbb{Z}$-graded algebra morphism $C(\mathcal{M})\leftarrow C(\mathcal{N})$.
\end{defi}
\noindent If in the previous definition the $\mathbb{Z}$-grading is replaced by an $\mathbb{N}$-grading and $E_{\bu}$ is a $\mathbb{Z}_{<0}$-graded vector bundle, one obtains the definition of an \textbf{N-manifold}. Because of Batchelor´s theorem for N-manifolds (see \cite[Thm.\,1]{Nman}), this definition is equivalent to the following (classical) definition of N-manifolds. 
\begin{defi}
    An N-manifold $\mathcal{M}=(M,C_\mathcal{M})$ is a smooth manifold $M$ endowed with a structure sheaf $C_{\mathcal{M}}$ which is a locally freely generated sheaf of $\mathbb{N}$-graded commutative unital $C^{\infty}_M$-algebras whose degree 0 term is $(C_{\mathcal{M}})_0=C^{\infty}_M$.
\end{defi}
\noindent As for morphisms between $N$-manifolds, there exists a 1-1-correspondence between $\mathbb{N}$-graded algebra morphisms $C_{\mathcal{M}}(M)\leftarrow C_{\mathcal{N}}(N)$ between the algebras of global sections and the following definition.
\begin{defi}
\label{Nmanmorsheaf}
    A morphism of N-manifolds $\phi=(\underline{\phi},\phi^{\ast}):\mathcal{M}\rightarrow \mathcal{N}$ is defined as a smooth map $\underline{\phi}:M\rightarrow N$, called the \textbf{underling map}, together with a morphism $\phi^{\ast}:C_{\mathcal{N}}\rightarrow \underline{\phi}_{\ast}C_{\mathcal{M}}$ of sheaves of graded $C^{\infty}_N$-algebras where smooth functions on $N$ act on $\underline{\phi}_{\ast}C_{\mathcal{M}}$ via the pullback $(\underline{\phi})^{\ast}$.
\end{defi}

\noindent Let $\catname{\mathbb{Z}_{<0}VB}, \catname{\mathbb{Z}VB}, \catname{\mathbb{N}Mfd}, \catname{\mathbb{Z}Mfd}$ denote the categories of $\mathbb{Z}_{<0}$-graded and $\mathbb{Z}$-graded vector bundles, N-manifolds and $\mathbb{Z}$-graded manifolds, respectively. There exist obvious functors $\mathfrak{S}:\catname{\mathbb{Z}_{<0}VB}\rightarrow \catname{\mathbb{N}Mfd}$ and $\mathfrak{S}:\catname{\mathbb{Z}VB}\rightarrow \catname{\mathbb{Z}Mfd}$ sending a graded vector bundle $E_{\bu}\rightarrow M$ to its associated graded manifold $\mathfrak{S}(E_{\bu})=(M, \Gamma S^{\bu }(E_{\bu}^{\ast}))$.
 For simplicity the latter will often also be denoted by $E_{\bu}$. Analogously, this notation applies for graded vector bundle morphisms.
\subsubsection{The infinitesimal object $D_{\bu}$}

\begin{defi}
   Let $D=(\star,\Gamma S^{\bu}((\mathbb{R}[-1])^{\ast}))$ denote the graded manifold associated to the graded vector bundle $\mathbb{R}[-1]$ over a point. Its graded function algebra $C(D)=\mathbb{R}[\epsilon]$ is generated by an element $\epsilon$ of \mbox{degree -1.}\par
   \noindent Further define the \s graded manifold $D_{\bu}$ as the “nerve of the pair groupoid of $D$". It is level-wise given by $D_n=D^{\times [n]}=D^{\times(n+1)},\:n\in \mathbb{N}_0$, its face and degeneracy maps $d_i^D,s_i^D:D_n\rightarrow D_{n\mp 1}$, $0\leq i\leq n$, forget and repeat the $i$th component, respectively.
\end{defi}
\begin{defi}
   For two graded manifolds $\mathcal{M},\mathcal{N}$ define the presheaf of graded manifolds $\text{Hom}(\mathcal{M},\mathcal{N})$ which sends a graded manifold $\mathcal{T}$ to the set $\text{Hom}(\mathcal{M},\mathcal{N})(\mathcal{T})=\hom_{\catname{\mathbb{Z}Mfd}}(\mathcal{M}\times \mathcal{T},\mathcal{N})$ and is given by post-composition on the level of morphisms.
\end{defi}
\begin{rema}
  The internal hom is an obvious bifunctor $\text{Hom}:\catname{\mathbb{Z}Mfd^{op}}\times \catname{\mathbb{Z}Mfd}\rightarrow \catname{pSH(\mathbb{Z}Mfd)}$, $(\mathcal{M},\mathcal{N})\mapsto \text{Hom}(\mathcal{M},\mathcal{N})$ which acts on the first and second component by pre-and postcomposition, respectively. For a 
fixed graded manifold $\mathcal{M}$, $\text{Hom}(D_{\bu},\mathcal{M})$ forms a co\s presheaf of graded manifolds.  
\end{rema}

\subsection{The tangent functor}
\begin{defi}\cite{severa, ZR}
   The tangent functor $\textswab{T}:\catname{sMfd}\rightarrow \catname{pSH}(\catname{\mathbb{Z}Mfd})$ assigns a \s manifold $X_{\bu}$ to the presheaf of graded manifolds
   \begin{gather*}
       \textswab{T}(X)=\text{Hom}(D_{\bu},X_{\bu}):\mathcal{T}\mapsto \text{Hom}(D_{\bu},X_{\bu})(\mathcal{T})\\
              = \left\{ (f_n)_{n\in \mathbb{N}_0}\in \bigsqcap_{n\in \mathbb{N}_0}\text{hom}_{\catname{\mathbb{Z}Mfd}}(D_n\times \mathcal{T},X_n)\ \middle\vert \begin{array}{l}
   f_n,n\in \mathbb{N}_0, \:\text{commute with the face and degeneracy maps}\\
    \text{of the \s graded manifolds}\:D_{\bu}\times \text{const}_{\mathcal{T}}\:\text{and}\:X_{\bu}
  \end{array}\right\}
   \end{gather*}
   where $const_{\mathcal{T}}$ denotes the constant \s graded manifold with value $\mathcal{T}$. The tangent functor acts by post-composition on the level of morphisms.
\end{defi}
\vspace{1.5mm}
\noindent Further define the following functors $H^k,\:k\in \mathbb{N}_0$, which “approximate” $\textswab{T}$ on finitely many levels.

\begin{defi}\cite{severa, ZR} Let $\star$ be the base point of $D= \mathbb{R}[-1]$. Set $H^0:\catname{sMfd}\rightarrow \catname{pSH}(\catname{\mathbb{Z}Mfd})$ to be the functor which assigns a \s manifold $X_{\bu}$ to the presheaf of graded manifolds $H^0(X)=\text{Hom}(\star, X_0)$ and is given by post-composition on the level of morphisms.\\
For $k\in \mathbb{N}$ define $H^k:\catname{sMfd}\rightarrow \catname{pSH}(\catname{\mathbb{Z}Mfd})$ to be the functor which assigns a \s manifold $X_{\bu}$ to the presheaf of graded manifolds 
\begin{gather*}
    H^k(X):\mathcal{T}\mapsto H^k(X)(\mathcal{T})\\
     = \left\{ (f_0,\dots,f_k)   \ \middle\vert \begin{array}{l}
f_l\in \text{hom}_{\catname{\mathbb{Z}Mfd}}(D_l\times \mathcal{T},X_l)\:\,\text{for}\:l<k,\:f_k\in \text{hom}_{\catname{\mathbb{Z}Mfd}}(\star \times D^k\times \mathcal{T},X_k),\\
f_0,\dots f_k \:\text{satisfy the conditions i)-iv)}
  \end{array}\right\}
\end{gather*}
and is given by post-composition on the level of morphisms. Setting $\iota$ to be the map $\iota:\star\rightarrow D$
and  $\iota_k:=\iota\times id_{D^k}:\star\times D^k\rightarrow D_k$, the morphisms $f_0,\dots,f_k$ satisfy the following conditions:\\
i)\,$(s_i^X\times id_{\mathcal{T}})f_l=f_{l+1}(s_i^D\times id_{\mathcal{T}})$ for $0\leq i\leq l\leq k-2$,\\
ii)\,$(d_i^X\times id_{\mathcal{T}})f_{l+1}=f_l(d_i^D\times id_{\mathcal{T}})$ for 
$0\leq i\leq l+1,\:l\leq k-2$,\\
iii)\,$(s_i^X\times id_{\mathcal{T}})f_{k-1}(\iota_{k-1}\times id_{\mathcal{T}})=f_k(s_i^D\times id_{\mathcal{T}})(\iota_{k-1}\times id_{\mathcal{T}})$ for\:\:$0\leq i\leq k-1$,\\
iv)\,$(d_i^X\times id_{\mathcal{T}})f_k=f_{k-1}(d_i^D\times id_{\mathcal{T}}) (\iota_k\times id_{\mathcal{T}})$ for $0\leq i\leq k$.
\end{defi}
\noindent We point out that the functor $H^k$ only sees 
the levels $0,\dots,k$ of a \s manifold, thus may be defined on the category of $l$-\s manifolds for $l\geq k$. 
 For $k\in \mathbb{N}_0$ there exist obvious natural transformations $p_{k+1,k}:H^{k+1}\rightarrow H^k$, $p_{\infty,k}:H\rightarrow H^k$ satisfying $p_{k+1,k}\circ p_{\infty,k+1}=p_{\infty,k}$ and $\textswab{T}$ indeed proves to be the limit of the diagram
\begin{gather*}
    \dots\rightarrow H^3\xrightarrow[p_{3,2}]{}H^2
    \xrightarrow[p_{2,1}]{}H^1\xrightarrow[p_{1,0}]{}H^0
\end{gather*}
(see \cite[Lemma\,3.5]{ZR}).
\noindent Moreover, for a \s manifold $X_{\bu}$ the presheaves $H^k(X),\:k\in \mathbb{N}$, can be recursively defined as fiber products in the following way.
\begin{prop}\label{fiberprod}\cite{ZR}
   Let $X_{\bu}$ be a \s manifold. For $k\in \mathbb{N}$ the presheaf $H^k(X)$ can be realized as the pullback
   \begin{center}
\begin{tikzcd}
H^k(X) \arrow[d, "{p_{k,k-1}^X}"] \arrow[rrrrrr, "F_k^X"]                                                                                                  &  &  &  &  &  & {Hom(\star\times D^k,X_k)} \arrow[d, "(s_0^D)^{\ast}\times \dots\times (s_{k-1}^D)^{\ast}\times (d_1^X)_{\ast}\times \dots\times (d_k^X)_{\ast}"] \\
H^{k-1}(X) \arrow[rrrrrr, "((s_0^X)_{\ast}\times \dots\times (s_{k-1}^X)_{\ast}\times (d_1^D)^{\ast}\times \dots\times (d_k^D)^{\ast})\circ F_{k-1}"'] &  &  &  &  &  & {Hom(\star\times D^{k-1},X_k)^{\times [k-1]}\times Hom(\star\times D^k,X_{k-1})^{\times [k]\setminus\{0\}}}                                      
\end{tikzcd}
   \end{center}
   where $F_l^X:H^l(X)\rightarrow \text{Hom}(\star\times D^l,X_l),\:l\in \mathbb{N}_0$, denotes the forgetful morphism only remembering the information on the highest level $\text{Hom}(\star\times D^l,X_l)$ of $H^l(X)$.
\end{prop}
\begin{rema}
\label{pullback}
    Because limits of functors are defined object-wise, there is an analogous statement about the recursive definition of the functors $H^k,\:k\in \mathbb{N}$. In particular, given a \s morphism $f_{\bu}:X_{\bu}\rightarrow Y_{\bu}$ between \s manifolds, $H^k(f)$ is the unique natural transformation $H^k(X)\rightarrow H^k(Y)$ satisfying $F^Y_k\circ H^k(f)=f_{\ast}\circ F_k^X$ and $p^Y_{k,k-1}\circ H^k(f)=H^{k-1}(f)\circ p^X_{k,k-1}$.
\end{rema} 
\noindent For a higher Lie groupoid $X_{\bu}$ the presheaves $H^0(X)$, $H^1(X)$ are canonically represented by $X_0$, $\text{ker}\,{Tp^1_0}_{|X_0}[1]$, respectively (see \cite[sec.\,3.2]{ZR}). In general, the presheaf $H^k(X)$ of a $k$-\s manifold $X_{\bu}$ satisfying suitable Kan conditions is non-canonically represented by the graded vector bundle consisting of the first $k$ levels of the tangent complex of $X_{\bu}$.

\begin{thm}\label{yoooo}\cite{ZR}\footnote{While not explicitly stated in \cite{ZR}, the presented statement can be deduced from Theorem 3.3, Lemma 3.8, and Lemma 3.18 in the same reference.}
  For $n\in \mathbb{N}_{\geq 2}$ let $X_{\bu}$ be an $n$-\s manifold satisfying Kan conditions $\text{Kan}(m,i)$ for $0\leq i\leq m$, $1\leq m\leq n$ (with horn projections being surjective submersions). Then the presheaves $H^l(X)$ are representable for $0\leq l\leq n$. Given a choice of compatible\footnote{A definition of compatible connections is given in section B.2 in the appendix, see also \cite[sec.\,2.5]{ZR}.\label{fnconn}} connections on $T^{\ast}X_k,\:T^{\ast}\bigwedge^k_0(X)$ w.r.t. $p^k_0:X_k\rightarrow \bigwedge^k_0(X)$ for $2\leq k\leq n$, one can construct (non-canonical) natural isomorphisms $H^k(X)\cong \text{Yon}(\bigoplus_{i=1}^k\text{ker}\,{Tp^i_0}_{|X_0}[i])$ under which the natural transformations $p^X_{k,k-1}:H^k(X)\rightarrow H^{k-1}(X)$ correspond to the projections $\bigoplus_{i=1}^k\text{ker}\,{Tp^i_0}_{|X_0}[i]\rightarrow \bigoplus_{i=1}^{k-1}\text{ker}\,{Tp^i_0}_{|X_0}[i]$. In particular, if $X_{\bu}$ is a Lie $\infty$-groupoid, the tangent functor $\textswab{T}(X)$ is representable and with a choice of compatible connections on $T^{\ast}X_k,\:T^{\ast}\bigwedge^k_0(X),\:k\geq 2$, one can construct a (non-canonical) natural isomorphism $\textswab{T}(X)\cong \text{Yon}( \bigoplus_{i=1}^{\infty}\text{ker}\,{Tp^i_0}_{|X_0}[i])$.
\end{thm}

\begin{rema}
The previous theorem is proven inductively in \cite{ZR}. The essential idea of the proof is as follows. Given compatible connections on $T^{\ast}X_k,\:T^{\ast}\bigwedge^k_0(X),\:k\geq 2$, and a natural isomorphism $H^{k-1}(X)\cong  \text{Yon}( \bigoplus_{i=1}^{k-1}\text{ker}\,{Tp^i_0}_{|X_0}[i])$, the diagram given in Proposition \ref{fiberprod} descends to a diagram of N-manifolds by Corollary \ref{moreD} and Lemma \ref{simpleD}. It is sufficient to show that $\bigoplus_{i=1}^{k}\text{ker}\,{Tp^i_0}_{|X_0}[i]$ is a limit of the latter. The resulting natural isomorphism $H^{k}(X)\cong \text{Yon}( \bigoplus_{i=1}^{k}\text{ker}\,{Tp^i_0}_{|X_0}[i])$ depends on the choice of connections and the specific way the fiber product $\bigoplus_{i=1}^{k}\text{ker}\,{Tp^i_0}_{|X_0}[i]$ is constructed in Lemma 3.18 of \cite{ZR}. 
\end{rema}

\subsection{Representability of the tangent functor}
We now want to find an analogous statement to Theorem \ref{yoooo} for the case of general \s manifolds. To achieve this, we employ the following approach.
\begin{enumerate}[label={\bfseries Step \arabic*:}]
\item At first we show that Theorem \ref{yoooo} still holds under the weaker assumption that horn projections are only submersions. In particular, it follows that for a Kan $n$-simplicial manifold $U_{\bu}$ which is induced by a \s manifold $X_{\bu}$, the presheaves $H^l(U),\:0\leq l\leq n$, are representable.
\item With $j_{U,X}:U_{\bu}\rightarrow X_{\bu}$ being the inclusion, we will show that $H^l(j_{U,X}):H^l(U)\rightarrow H^l(X),\:0\leq l\leq n$, are natural isomorphisms. In particular, the presheaves $H^n(X),n\in \mathbb{N}_0$, are representable and only depend on the local structure of $X_{\bu}$ near the base space $X_0$.\\
We proceed by making the following substeps:
\begin{enumerate}[label={\bfseries Step 2\alph*:}]
\item proving that $H^l(j_{U,X}):H^l(U)\rightarrow H^l(X),\:0\leq l\leq n$, are natural isomorphisms if $U_{\bu}$ satisfies a certain extra assumption,
\item proving that for Kan $n$-\s manifolds $\tilde{U}_{\bu}\subseteq U_{\bu}$ induced by $X_{\bu}$, 
$H^l(j_{\Tilde{U},U}):H^l(\tilde{U})\rightarrow H^l(U),\:0\leq l\leq n$, are natural isomorphisms where $j_{\Tilde{U},U}$ denotes the inclusion $\Tilde{U}_{\bu}\hookrightarrow U_{\bu}$,
\item proving the statement by combining the two previous subresults.
\end{enumerate}
\item As the tangent functor $\textswab{T}(X)$ is the limit of representable presheaves $H^n(X),\:n\in \mathbb{N}_0$, it will prove to be representable by the N-manifold associated to the graded vector bundle $\bigoplus_{i=1}^{\infty} \text{ker}\,{Tp^i_0}_{|X_0}[i]$. 
\end{enumerate}
\noindent\textbf{(Step 1)}
    \begin{lemm}
    \label{weakercase}
    Let $X_{\bu}$ be a Kan $n$-\s manifold\footnote{Note that the essential difference to Theorem \ref{yoooo} is that horn projections are assumed to be submersions rather than surjective submersions.}. Then $H^l(X)$ is representable for $0\leq l\leq n$. Given a choice of compatible\footref{fnconn} connections on $T^{\ast}X_k,T^{\ast}\bigwedge^k_0(X)$ w.r.t. $Tp^k_0:TX_k\rightarrow T\bigwedge^k_0(X)$ for $2\leq k\leq n $, one can construct natural isomorphisms $H^k(X)\cong \text{Yon}(\bigoplus_{i=1}^k\,\text{ker}\,{Tp^i_0}_{|X_0}[i])$.
    \end{lemm}
    
     \begin{proof}
    We prove this inductively. As a base case, we first show that $H^1(X)$ is naturally represented by $\text{ker}\,{Tp^1_0}_{|X_0}[1]=\text{ker}\,{Td_1}_{|X_0}[1]$ for any Kan 1-\s manifold.
    By Proposition \ref{fiberprod} and Lemma \ref{simpleD}, $H^1(X)$ can be defined as a fiber product of presheaves that are represented by $\mathbb{Z}_{<0}$-graded vector bundles and the involved morphisms correspond to $\mathbb{Z}_{<0}$-graded vector bundle morphisms. Since the functor $\mathfrak{S}:\catname{\mathbb{Z}_{<0}VB}\rightarrow \catname{\mathbb{N}Mfd}$ which sends a $\mathbb{Z}_{<0}$-graded vector bundle to its associated N-manifold 
    preserves limits (see \cite[Prop.\,2.25]{ZR}), it is enough to prove that $\text{ker}\,{Td_1}_{|X_0}[1]$ is the fiber product of the following diagram in the category of $\mathbb{Z}_{<0}$-graded vector bundles 
    \begin{center}
\begin{tikzcd}
{\text{ker}\,{Td_1}_{|X_0}[1]} \arrow[d, dashed] \arrow[rr, dashed, hook] &  & {TX_1[1]} \arrow[d, "{0\times Td_1[1]}"] \\
X_0 \arrow[rr, "s_0\times 0"']                                            &  & {X_1\times TX_0[1]}                     
\end{tikzcd}
    \end{center}
    where $X_0,\:X_1$ denote the trivial graded vector bundle over $X_0,\:X_1$, respectively. To show that ${\text{ker}\,{Td_1}_{|X_0}[1]}$ satisfies the universal property of the fiber product, 
    consider
    a graded vector bundle $E_{\bu}\rightarrow M$ and graded vector bundle morphisms $(F,f):E_{\bu}\rightarrow X_0$, $(G,g):E_{\bu}\rightarrow TX_1[1]$. Then $(s_0\times 0)\circ F=(0\times Td_1[1])\circ G$ if and only if $G$ maps to $\text{ker}\,Td_1[1]\subseteq TX_1[1]$ and $s_0f=g$, which in turn is equivalent to the fact that $G$ factors (uniquely) through the pullback $(s_0)^{\ast}(\text{ker}\,Td_1[1])={\text{ker}\,Td_1}_{|X_0}[1]$. So ${\text{ker}\,Td_1}_{|X_0}[1]$ is indeed the fiber product. \\
    Under the weaker the assumption that horn projections are submersions (instead of surjective submersions), the rest of the inductive proof of Theorem \ref{yoooo} in \cite{ZR} can be straightforwardly extended, with the exception of Lemma 3.16 in \cite{ZR}. For the case that horn projections are submersions, Lemma 3.16 is proven in Lemma \ref{proofanhang} in the appendix.
    \end{proof}
    
    \begin{rema}
    \label{uniqueway}
    Given a Kan $n$-\s manifold $X_{\bu}$ and compatible connections like in Lemma \ref{weakercase}, we from now on always work with the specific fiber products $\bigoplus_{i=1}^k\,\text{ker}\,{Tp^i_0}_{|X_0}[i]$ and natural isomorphisms $H^k(X)\cong \text{Yon}(\bigoplus_{i=1}^k\,\text{ker}\,{Tp^i_0}_{|X_0}[i])$ constructed in the proof of Lemma 3.18 in \cite{ZR}.
    \end{rema}
    \vspace{1mm}
    \noindent We point out the following immediate consequence of Lemma \ref{weakercase}.
    
    \begin{coro}
    Let $X_{\bu}$ be a \s manifold and let $U_{\bu}$ be a Kan $n$-\s manifold induced by $X_{\bu}$. Then $H^l(U)$ is representable for $0\leq l\leq n$. Given a choice of compatible connections on $T^{\ast}U_k$, $T^{\ast}\bigwedge^k_0(U)$ for $2\leq k\leq n $, one can construct natural isomorphisms $H^k(U)\cong\text{Yon}(\bigoplus_{i=1}^k\,\text{ker}\,{Tp^i_0}_{|X_0}[i])$.
    \end{coro}

   \noindent \textbf{(Step 2)}\:\: Denoting $j:U_{\bu}\rightarrow X_{\bu}$ to be the inclusion, in a series of consecutive lemmas, we now will prove that $H^l(j):H^l(U)\rightarrow H^l(X)$ is a natural isomorphism. Initially, we will prove this for a special case in the next lemma.\\
    
       \noindent \textbf{(Step 2a)}
       \begin{lemm}
       \label{UX}
    Let $X_{\bu}$ be a \s manifold, $U_{\bu}$ be a Kan $n$-\s mani\-fold induced by $X_{\bu}$ and let $j:U_{\bu}\rightarrow X_{\bu}$ denote the inclusion. Suppose that there exist compatible connections on $T^{\ast}X_k,\:T^{\ast}U_k,T^{\ast}\bigwedge^k_0(U)$ w.r.t. $Tj_k:TU_k\hookrightarrow TX_k$ and $Tp^k_0:TU_k\rightarrow T\bigwedge^k_0(U)$ for $2\leq k\leq n$. Then $H^l(j):H^l(U)\rightarrow H^l(X)$ is a natural isomorphism for $0\leq l\leq n$.
    \end{lemm}
    
     \begin{proof}
    We first show that $H^1(X)$ is representable (without needing to choose any connections) and that $H^1(j):H^1(U)\rightarrow H^1(X)$ is a natural isomorphism. Similarly to the proof of Lemma \ref{weakercase}, the diagram
    \begin{center}
\begin{tikzcd}
{{\text{ker}\,Tp^1_0}_{|X_0}[1]} \arrow[d] \arrow[rr, hook] &  & {TU_1[1]} \arrow[r, "{Tj_1[1]}", hook] & {TX_1[1]} \arrow[d, "{0\times Td_1[1]}"] \\
X_0 \arrow[rrr, "s_0\times 0"']                             &  &                                        & {X_1\times TX_0[1]}                     
\end{tikzcd}
    \end{center}
    proves to be a pullback diagram in the category of $\mathbb{Z}_{<0}$-graded vector bundles and corresponds to a pullback diagram in the category of presheaves of N-manifolds.
    Let $H^1(U)\cong \text{Yon}({\text{ker}\,Tp^1_0}_{|X_0}[1])$ be the unique natural isomorphism making the diagrams 
    \begin{center}
\begin{tikzcd}
{\text{Yon}({\text{ker}\,Tp^1_0}_{|X_0}[1])} \arrow[d] \arrow[r, "\cong"] & H^1(U) \arrow[d, "{p^U_{1,0}}"] & {\text{Yon}({\text{ker}\,Tp^1_0}_{|X_0}[1])} \arrow[r, "\cong"] \arrow[d] & H^1(U) \arrow[d, "F^U_1"] \\
\text{Yon}(X_0) \arrow[r, "\cong"']                                       & H^0(U)                          & {\text{Yon}(TU_1[1])} \arrow[r, "\cong"']                                 & {Hom(\star\times D,U_1)} 
\end{tikzcd}
    \end{center}
    commute. Analogously define the natural isomorphism $H^1(X)\cong \text{Yon}({\text{ker}\,Tp^1_0}_{|X_0}[1])$. Due to the fact that $H^1(j)$ is the unique natural transformation $H^1(U)\rightarrow H^1(X)$ satisfying $p^X_{1,0}\circ H^1(j)=p^U_{1,0}$, $j_{\ast}\circ F^U_1=F^X_1\circ H^1(j)$\footnote{see Remark \ref{pullback}}, it can be easily verified via diagram chasing that the isomorphism $H^1(U)\xrightarrow[]{\cong}\text{Yon}({\text{ker}\,Tp^1_0}_{|X_0}[1])\xrightarrow[]{\cong}H^1(X)$ equals $H^1(j):H^1(U)\rightarrow H^1(X)$.\\
    For a \s manifold $X_{\bu}$ and an induced Kan $n$-\s manifold $U_{\bu}$ as in the hypothesis we now prove the following statement inductively.\\
    \underline{Claim}: For $1\leq l\leq n$ the presheaf $H^l(X)$ can be represented by $\bigoplus_{i=1}^l\,\text{ker}\,{Tp^i_0}_{|X_0}[i]$. Let 
    \begin{center}
\begin{tikzcd}
{\bigoplus_{i=1}^l\,\text{ker}\,{Tp^i_0}_{|X_0}[i]} \arrow[rrr, "F^U_l"] \arrow[d, "{p^U_{l,l-1}}"'] &  &  & {T[1]^lU_l} \arrow[d]                                            \\
{\bigoplus_{i=1}^{l-1}\,\text{ker}\,{Tp^i_0}_{|X_0}[i]} \arrow[rrr]                                  &  &  & {{(T[1]^{l-1}U_l)^{[l-1]}\times(T[1]^lU_{l-1})^{[l]\setminus0}}}
\end{tikzcd}

    \end{center}
    be the diagram of N-manifolds which corresponds\footnote{by Lemma \ref{weakercase} and Corollary \ref{rep2}\label{refnote}} to the pullback diagram of presheaves defining $H^l(U)$. Then 
    \begin{center}
\begin{tikzcd}
{\bigoplus_{i=1}^l\,\text{ker}\,{Tp^i_0}_{|X_0}[i]} \arrow[d, "{p^U_{l,l-1}}"'] \arrow[rrr, "{T[1]^lj_l\circ F^U_l}"] &  &  & {T[1]^lX_l} \arrow[d]                                            \\
{\bigoplus_{i=1}^{l-1}\,\text{ker}\,{Tp^i_0}_{|X_0}[i]} \arrow[rrr]                                                                     &  &  & {{(T[1]^{l-1}X_l)^{[l-1]}\times(T[1]^lX_{l-1})^{[l]\setminus0}}}
\end{tikzcd}
    \end{center}
    is a pullback diagram in the category of N-manifolds and corresponds to the pullback diagram of presheaves defining $H^l(X)$.
    With the isomorphisms $H^l(U)\cong \text{Yon}(\bigoplus_{i=1}^l\,\text{ker}\,{Tp^i_0}_{|X_0}[i])$\footnote{Please pay attention to Remark \ref{uniqueway}.}, $H^l(X)\cong \text{Yon}(\bigoplus_{i=1}^l\,\text{ker}\,{Tp^i_0}_{|X_0}[i])$ defined in a similar way to the above isomorphisms $H^1(U)\cong \text{Yon}({\text{ker}\,Tp^1_0}_{|X_0}[1])$, $H^1(X)\cong \text{Yon}({\text{ker}\,Tp^1_0}_{|X_0}[1])$, the composition $H^l(U)\xrightarrow[]{\cong}\text{Yon}(\bigoplus_{i=1}^l\,\text{ker}\,{Tp^i_0}_{|X_0}[i])\xrightarrow[]{\cong}H^l(X)$ is equal to $H^l(j):H^l(U)\rightarrow H^l(X)$.\\
    \underline{Proof of the claim}: We already proved this for $l=1$. Assuming that the statement is true until $l-1$, the following diagram of N-manifolds commutes
    \begin{center}
\begin{tikzcd}
{\bigoplus_{i=1}^{l}\,\text{ker}\,{Tp^i_0}_{|X_0}[i]} \arrow[d, "{p^U_{l,l-1}}", dashed] \arrow[rr, "F^U_l", dashed]    &                                                                                                              & {T[1]^lU_l} \arrow[r, "{T[1]^lj_l}", hook] \arrow[rdd, dashed] & {T[1]^lX_l} \arrow[d, "(s_I^D)^{\ast}\times (d_J^X)_{\ast}"]                   \\
{\bigoplus_{i=1}^{l-1}\,\text{ker}\,{Tp^i_0}_{|X_0}[i]} \arrow[rd, "F_{l-1}^U"', dashed] \arrow[r, "F^{l-1}_X"] & {T[1]^{l-1}X_{l-1}} \arrow[rr, "(s_I^X)_{\ast}\times(d_J^D)^{\ast}"]                                         &                                                        & {(T[1]^{l-1}X_l)^{\times[l-1]}\times(T[1]^lX_{l-1})^{\times[l]\setminus0}}                 \\
                                                                                                        & {T[1]^{l-1}U_{l-1}} \arrow[u, "{T[1]^{l-1}j_{l-1}}"', hook] \arrow[rr, "(s_I^U)_{\ast}\times(d_J^D)^{\ast}", dashed] &                                                        & {(T[1]^{l-1}U_l)^{\times[l-1]}\times(T[1]^lU_{l-1})^{\times[l]\setminus0}} \arrow[u, hook]
\end{tikzcd}
    \end{center}
    where $I,J$ are the tuples $I:=(0,\dots,l-1)$, $J:=(1,\dots,l)$ and products as $s_0^X\times \dots  \times s_{l-1}^X$, $d_1^X\times \dots \times  d_l^X$ are abbreviatedly denoted by $s_I^X$, $d_J^X$. Note that the dashed arrows form a pullback diagram by Lemma \ref{weakercase}.\\
    To show that the upper rectangle is a pullback diagram, i.e. $(\bigoplus_{i=1}^{l}\,\text{ker}\,{Tp^i_0}_{|X_0}[i],p^U_{l,l-1},T[1]^lj_l\circ F^U_l)$ satisfies the universal property of the fiber product, suppose that there exist an N-manifold $\mathcal{T}$ and maps  $\alpha:\mathcal{T}\rightarrow \bigoplus_{i=1}^{l-1}\,\text{ker}\,{Tp^i_0}_{|X_0}[i]$, $\beta:\mathcal{T}\rightarrow T[1]^lX_l$ such that $((s_I^X)_{\ast}\times (d_J^D)^{\ast})\circ F_{l-1}^X\circ \alpha=((s_I^D)^{\ast}\times (d_J^X)_{\ast})\circ \beta$.
    By induction hypothesis, one has $((s_I^X)_{\ast}\times (d_J^D)^{\ast})\circ F_{l-1}^X\circ \alpha=((T[1]^{l-1}j_l)^{[l-1]}\times (T[1]^lj_{l-1})^{[l]\setminus0})\circ ((s_I^U)_{\ast}\times (d_J^D)^{\ast})\circ F_{l-1}^U\circ \alpha$. Considering the morphisms of N-manifolds as sheaf morphisms (according to Def. \ref{Nmanmorsheaf}), it follows for the underlying smooth maps that
    \begin{gather*}
        \underline{T[1]^{l-1}j_l}\circ \underline{(s_i^U)_{\ast}}\circ \underline{F_{l-1}}\circ \underline{\alpha}=\underline{(s_i^X)_{\ast}}\circ \underline{F_{l-1}}\circ \underline{\alpha}=\underline{(s_i^D)^{\ast}}\circ \underline{\beta}
    \end{gather*}
    for $i\in \{0,\dots,l-1\}$. Since $\underline{T[1]^{l-1}j_l}=j_l, \underline{(s_i^U)_{\ast}}=s_i^U, \underline{(s_i^D)^{\ast}}=id_{X_l}$, the underlying map of $\beta$ is given by $\underline{\beta}=j_l\circ s_i^U\circ \underline{F_{l-1}}\circ \underline{\alpha}$ and in particular maps to $U_l$. Thus the codomain of $\beta$ can be restricted to $T[1]^lU_l$ and $\beta$ factors through a morphism $\tilde{\beta}:\mathcal{T}\rightarrow T[1]^lU_l$, i.e. $\beta=T[1]^lj_l\circ \tilde{\beta}$.\\
    Because 
    $(T[1]^{l-1}j_l)^{[l-1]}\times (T[1]^lj_{l-1})^{[l]\setminus0}$ is a monomorphism
    and
    \begin{gather*}
        ((T[1]^{l-1}j_l)^{[l-1]}\times (T[1]^lj_{l-1})^{[l]\setminus0})\circ ((s_I^D)^{\ast}\times (d_J^U)_{\ast})\circ \Tilde{\beta}=((s_I^D)^{\ast}\times(d_J^X)_{\ast})\circ T[1]^lj_l\circ \tilde{\beta}\\=((s_I^D)^{\ast}\times (d_J^X)_{\ast})\circ \beta=
        ((T[1]^{l-1}j_l)^{[l-1]}\times (T[1]^lj_{l-1})^{[l]\setminus0})\circ ((s_I^U)_{\ast}\times (d_J^D)^{\ast})\circ F_{l-1}^U\circ \alpha,
    \end{gather*}
    it follows that $((s_I^U)_{\ast}\times (d_J^D)^{\ast})\circ F_{l-1}^U\circ \alpha=((s_I^D)^{\ast}\times (d_J^U)_{\ast})\circ \tilde{\beta}$.
    Due to the universal property of the fiber product $(\bigoplus_{i=1}^{l}\,\text{ker}\,{Tp^i_0}_{|X_0}[i],p^U_{l,l-1},F^U_l)$, there exists a unique map $\gamma:\mathcal{T}\rightarrow \bigoplus_{i=1}^{l}\,\text{ker}\,{Tp^i_0}_{|X_0}[i]$ such that $\alpha=p_{l,l-1}^U\circ \gamma,\:\tilde{\beta}=F_l^U\circ \gamma$, and it is clear that $\gamma$ is also the unique map such that $\alpha=p_{l,l-1}^U\circ \gamma$, $\beta=T[1]^lj_l\circ F_l^U\circ \gamma$.\\
    Similarly like before one can use the fact that $H^l(j)$ is the unique natural transformation $H^l(U)\rightarrow H^l(X)$ satisfying $p^X_{l,l-1}\circ H^l(j)=j_{\ast}\circ p^U_{l,l-1}$, $F^X_l\circ H^l(j)=j_{\ast}\circ F^U_l$ to verify that the isomorphism $H^l(U)\xrightarrow[]{\cong}\text{Yon}(\bigoplus_{i=1}^{l}\,\text{ker}\,{Tp^i_0}_{|X_0}[i])\xrightarrow[]{\cong}H^l(X)$ equals $H^l(j)$.
    \end{proof}
  
  
       \noindent \textbf{(Step 2b)}
    \begin{lemm}
    \label{UU}
    Let $\tilde{U}_{\bu}\subseteq U_{\bu}$ be Kan $n$-\s manifolds which are induced by a \s manifold $X_{\bu}$.
    Let $j:\tilde{U}_{\bu}\rightarrow U_{\bu}$ denote the inclusion.
    Then $H^l(j):H^l(\tilde{U})\rightarrow H^l(U)$ is a natural isomorphism for $0\leq l\leq n$.
    \end{lemm}
    
     \begin{proof}
     The statement obviously holds for $l=0$. In the case $l=1$ it can be proven in the same way as in the proof of Lemma \ref{UX}. Now 
     choose compatible connections on $T^{\ast}U_k,\:T^{\ast}\bigwedge^k_0(U)$, $2\leq k\leq n$, and restrict them to connections on $T^{\ast}\tilde{U}_k,\:T^{\ast}\bigwedge^k_0(\tilde{U})$.
     By Lemma \ref{weakercase} and Remark \ref{uniqueway} let $H^k(\tilde{U})\cong \text{Yon}(\bigoplus_{i=1}^{k}\,\text{ker}\,{Tp^i_0}_{|X_0}[i])$, $H^k(U)\cong \text{Yon}(\bigoplus_{i=1}^{k}\,\text{ker}\,{Tp^i_0}_{|X_0}[i])$ be the specific natural isomorphisms depending on the respective connections for $\tilde{U}_{\bu},U_{\bu}$. Due to the way the fiber products $\bigoplus_{i=1}^{l}\,\text{ker}\,{Tp^i_0}_{|X_0}[i]$, $0\leq l\leq n$, are defined (on the level of graded manifolds) (see Lemma 3.18. in \cite{ZR}), it can be easily checked that one has $T[1]^lj_l\circ F^{\tilde{U}}=F^U_l$ (and the natural morphisms $p^{\tilde{U}}_{l,l-1},p^U_{l,l-1}:\bigoplus_{i=1}^{l}\,\text{ker}\,{Tp^i_0}_{|X_0}[i])\rightarrow \bigoplus_{i=1}^{l-1}\,\text{ker}\,{Tp^i_0}_{|X_0}[i])$  trivially coincide).
    As in the proof of Lemma \ref{UX} it can be verified that the composition $H^l(\tilde{U})\xrightarrow[]{\cong}$ $\text{Yon}(\bigoplus_{i=1}^{l}\,\text{ker}\,{Tp^i_0}_{|X_0}[i])\xrightarrow[]{\cong}H^l(U)$ is equal to $H^l(j):H^l(\tilde{U})\rightarrow H^l(U)$.
    \end{proof}
       \noindent \textbf{(Step 2c)}\:\: To combine the results of Lemma \ref{UX} and \ref{UU}, we require the following lemma.
       \begin{lemm}
    \label{lemmconn}
    Let $X_{\bu}$ be a \s manifold, $U_{\bu}$ be an $n$-\s manifold induced by $X_{\bu}$ and $\eta_k$ be connections on $T^{\ast}U_k$ for $2\leq k\leq n$. Then there exist a Kan $n$-\s manifold $\tilde{U}_{\bu}\subseteq U_{\bu}$ induced by $X_{\bu}$ and connections $\psi_k$ on $T^{\ast}X_k$, $2\leq k\leq n$, such that ${\psi_k}_{|\tilde{U}_k}={\eta_k}_{|\tilde{U}_k}$.
    \end{lemm}
    
     \begin{proof}
    For a \s manifold $X_{\bu}$ the embeddings $X_0\hookrightarrow X_m,\:m\in \mathbb{N}$, are proper because they have smooth left inverses (see \cite[Prop.\,A.53\,d)]{lee}), thus $X_0\subseteq X_m$ is a closed embedded submanifold.  Since $U_m$ is a normal space, there exist open neighborhoods $V_m$ of $X_0$ such that $X_0\subseteq V_m\subseteq \Bar{V_m}\subseteq U_m$. Using the constructions made in the proofs of Theorem \ref{yummiiiiiiiiii} and Lemma \ref{blado}, one can find a Kan $n$-\s manifold $\tilde{U}_{\bu}\subseteq U_{\bu}$ such that $\tilde{U}_k$ is an open subset of $V_k$ for $2\leq k\leq n$. Lastly, applying a partition of unity subordinate to the open cover $\{U_k, X_k\setminus\widebar{\tilde{U_k}}$\} of $X_k$ one can define a connection $\psi^k$ on $T^{\ast}X_k$ satisfying $\psi^k_{|\tilde{U}_k}=\eta^k_{|\tilde{U}_k}$.
    \end{proof}

   \noindent Thus, by Lemma \ref{lemmconn} one can always find an induced Kan $n$-\s manifold such that the assumptions of Lemma \ref{UX} are satisfied. We can now conclude:
    \begin{coro}
    \label{UXgen}
    Let $X_{\bu}$ be a \s manifold and $U_{\bu}$ be an induced Kan $n$-\s manifold. Let $j:U_{\bu}\rightarrow X_{\bu}$ be the inclusion. Then $H^l(j):H^l(U)\rightarrow H^l(X)$ is a natural isomorphism for $0\leq l\leq n$. In particular, the presheaves $H^k(X),\:k\in \mathbb{N}_0$, are representable by $\bigoplus_{i=1}^k\,\text{ker}\,{Tp^i_0}_{|X_0}[i]$.
    \end{coro}
     \begin{proof}
     Choose compatible connections $\eta^k,\rho^k,\:2\leq k\leq n$, on $T^{\ast}U_k,\:T^{\ast}\bigwedge^k_0(U)$. By Lemma \ref{lemmconn} and Remark \ref{remaconn} i) there exist an induced Kan $n$-\s manifold $\tilde{U}_{\bu}\subseteq U_{\bu}$ and connections $\psi^k$ on $T^{\ast}X_k$ such that $\psi^k,\:\eta^k_{|\tilde{U_k}},\:\phi^k_{|\tilde{U}_k}$ are compatible connections on
     $T^{\ast}X_k,\:T^{\ast}\tilde{U}_k,\:T^{\ast}\bigwedge^k_0(\tilde{U})$ for $2\leq k\leq n$. Let the inclusions $U_{\bu}\hookrightarrow X_{\bu},\:\tilde{U}_{\bu}\hookrightarrow X_{\bu},\:\tilde{U}_{\bu}\hookrightarrow U_{\bu}$ be denoted by $j_{U,X},\:j_{\tilde{U},X},\:j_{\tilde{U},U}$, respectively. Since according to Lemma \ref{UX} and \ref{UU} the natural transformations $H^l(j_{\tilde{U},X}),\:H^l(j_{\tilde{U},U}),\:0\leq l\leq n$, are natural isomorphisms and $H^l(j_{\tilde{U},X})=H^l(j_{U,X})\circ H^l(j_{\tilde{U},U})$, the morphisms $H^l(j_{U,X})$ are natural isomorphisms as well.
    \end{proof}

         \noindent \textbf{(Step 3)} Finally, we can present the main result of this paper.
  \begin{thm}
    Let $X_{\bu}$ be a \s manifold, then the tangent functor $\textswab{T}(X)$ is representable. Given an induced Kan semi\s manifold $U_{\bu}$ and a choice of compatible connections on $T^{\ast}U_k,T^{\ast}\bigwedge^k_0(U)$, $k\geq 2$, one can construct an isomorphism
    \begin{gather*}
        \textswab{T}(X)\cong \text{Yon}(\bigoplus_{i=1}^{\infty}\text{ker}\,{Tp^i_0}_{|X_0}[i]).
    \end{gather*}
    
    \end{thm}
    \begin{proof}
    For $n\in \mathbb{N}$ there exists an induced Kan $n$-\s manifold $U^{(n)}_{\bu}$ such that $U^{(n)}_m$ is an open subset of $U_m$ for $m\in \{0,\dots n\}$. For $2\leq k\leq n$ restrict the connections on $T^{\ast}U_k$ and $T^{\ast}\bigwedge^k_0(U)$ to connections on $T^{\ast}U^{(n)}_k$ and $T^{\ast}\bigwedge^k_0(U^{(n)})$. By Lemma \ref{weakercase} and Corollary \ref{UXgen} one has $\text{Yon}(\bigoplus_{i=1}^l\text{ker}\,{Tp^i_0}_{|X_0}[i])\cong H^l(U^{(n)})\cong H^l(X)$ for $0\leq l\leq n$. Note that the isomorphisms $\text{Yon}(\bigoplus_{i=1}^l\text{ker}\,{Tp^i_0}_{|X_0}[i])\cong H^l(X)$ do not depend on the choice of $U_{\bu}^{(n)}$ if one consistently utilizes the particular natural isomorphisms specified in Remark \ref{uniqueway} (see also Remark \ref{choice}).
  It follows that the sequence of presheaves
\begin{gather*}
  \dots\rightarrow H^k(X)\xrightarrow[]{p^X_{k,k-1}} H^{k-1}(X)\rightarrow \dots  
\end{gather*}
corresponds to the sequence of N-manifolds
\begin{gather*}
  \dots\rightarrow \bigoplus_{i=1}^k \text{ker}\,{Tp^i_0}_{|X_0}[i]\rightarrow  
  \bigoplus_{i=1}^{k-1} \text{ker}\,{Tp^i_0}_{|X_0}[i]\rightarrow \dots
\end{gather*}
Thus, $\textswab{T}(X)=\lim_{k}H^k(X)\cong\lim_{k} \text{Yon}(\bigoplus_{i=1}^k \text{ker}\,{Tp^i_0}_{|X_0}[i])=\text{Yon}(\bigoplus_{i=1}^{\infty} \text{ker}\,{Tp^i_0}_{|X_0}[i])$.
    \end{proof}
    \begin{rema}
    \label{choice}
    Under the assumption that one consistently uses the particular natural isomorphisms specified in Remark \ref{uniqueway}, the isomorphism $ \textswab{T}(X)\cong \text{Yon}(\bigoplus_{i=1}^{\infty}\text{ker}\,{Tp^i_0}_{|X_0}[i])$ given in the previous theorem only depends on “germs of connections". To formulate this precisely, let us introduce the following equivalence relation. For $U_{\bu},\:\tilde{U}_{\bu}$ being Kan semi\s manifolds induced by a fixed \s manifold $X_{\bu}$, two collections of degree-wise compatible connections $(\eta^k,\rho^k)_{k\geq 2}$, $(\tilde{\eta}^k,\tilde{\rho}^k)_{k\geq 2}$ defined on $T^{\ast}U_k,\:T^{\ast}\bigwedge^k_0(U)$ and $T^{\ast}\tilde{U}_k,\:T^{\ast}\bigwedge^k_0(\tilde{U})$, $k\geq 2$, are equivalent if their restrictions agree on an induced Kan semi\s manifold $U_{\bu}^{\prime}\subseteq U_{\bu}\cap \tilde{U}_{\bu}$.\\
    Given a \s manifold $X_{\bu}$, induced Kan semi\s manifolds $U_{\bu}$, $\tilde{U}_{\bu}$ with two equivalent collections of connections and Kan $n$-\s submanifolds $U^{(n)}_{\bu}\subseteq U_{\bu}$, $\tilde{U}^{(n)}_{\bu}\subseteq \tilde{U}_{\bu}$ like in the proof of the previous theorem, there exists an induced Kan $n$-\s manifold ${U^{\prime}}^{(n)}_{\bu}\subseteq U^{(n)}_{\bu}\cap {\tilde{U}}^{(n)}_{\bu}$ such that the following diagram commutes for $0\leq l\leq n$.
    \begin{center}
\begin{tikzcd}
                        & H^l(X)                                                                                                                    &                                 \\
H^l(U^{(n)}) \arrow[ru] & {\text{Yon}(\bigoplus_{i=1}^{l}\text{ker}\,{Tp^i_0}_{|X_0}[i])} \arrow[u, "\cong"] \arrow[l, "\cong"'] \arrow[r, "\cong"] & H^l(\tilde{U}^{(n)}) \arrow[lu] \\
                        & H^l({U^{\prime}}^{(n)}) \arrow[u, "\cong"] \arrow[lu] \arrow[ru]                                                          &                                
\end{tikzcd}
    \end{center}
Here, the particular isomorphisms $H^l({U^{\prime}}^{(n)})\cong {\text{Yon}(\bigoplus_{i=1}^{l}\text{ker}\,{Tp^i_0}_{|X_0}[i])}$ (specified in Remark \ref{uniqueway}) are induced by the restricted connections on ${U^{\prime}_{\bu}}^{(n)}$.
 \noindent     Consequently, the natural isomorphisms $ \textswab{T}(X)\cong \text{Yon}(\bigoplus_{i=1}^{\infty}\text{ker}\,{Tp^i_0}_{|X_0}[i])$ induced by the two choices of connections for $\tilde{U}_{\bu},U_{\bu}$ coincide.
    
\end{rema}

\section{Examples}
\noindent Let us first present an example which illustrates how simplicial manifolds may fail to be Kan.
\begin{example}
We gradually construct a \s manifold whose $(m,i)$-horn spaces do not exist for $m\geq 2$ and thus cannot satisfy any Kan conditions except the 1-Kan conditions. We start with the 1-\s manifold $X_{\bu}^{(1)}$ 
    whose terms are $X^{(1)}_0=\mathbb{R}$ and $X^{(1)}_1=\mathbb{R}_{(0)}\sqcup \mathbb{R}_{(1)}$, where $X^{(1)}_1$ is defined as the disjoint union of two copies of $\mathbb{R}$. The degeneracy map $s_0:X_0^{(1)}\rightarrow X_1^{(1)}$ maps identically to the first copy $\mathbb{R}_{(0)}$. The face maps $d_0,d_1:X_1^{(1)}\rightarrow X_0^{(1)}$ send $\mathbb{R}_{(0)}$ identically to $X_0$ and map $\mathbb{R}_{(1)}$ to $0$. With $\text{sk}^1(X^{(1)})$ being the 1-skeleton of $X^{(1)}_{\bu}$, define the 2-\s manifold $X^{(2)}_{\bu}$ given by $X^{(2)}_0=X^{(1)}_0,\:X^{(2)}_1=X^{(1)}_1,\:X^{(2)}_2=(\text{sk}^1(X^{(1)}))_2\sqcup \mathbb{R}_{(2)}$. The 2-face maps of $X^{(2)}_{\bu}$ map $\mathbb{R}_{(2)}\subseteq X_2^{(2)}$ to $0\in \mathbb{R}_{(1)}$, apart from that the structure maps of $X_{\bu}^{(2)}$ coincide with the structure maps of the 2-truncation $\text{tr}_2\text{sk}^1(X^{(1)})$. Repeating the above procedure, i.e. recursively defining the $(n+1)$-\s manifold $X^{(n+1)}_{\bu}$ by adding a disjoint copy $\mathbb{R}_{(n+1)}$ to $(\text{tr}_{n+1}\text{sk}^n(X^{(n)}))_{n+1}$ which is mapped by degeneracy maps to $0\in \mathbb{R}_{(n)}$, one obtains a sequence of truncated \s manifolds $(X^{(n)}_{\bu})_{n\in \mathbb{N}}$. Let $X_{\bu}$ be the \s manifold whose $n$-truncation is equal to $X^{(n)}_{\bu}$. Then the $(m,i)$-horn spaces of $X_{\bu}$ do not exist as they are (setwise) disjoint unions of manifolds of different dimension. 
\end{example}
\noindent As mentioned in the introduction, it is quite challenging to find interesting examples of non-Kan simplicial manifolds. Therefore, we confine ourselves to illustrating how Theorem \ref{yummiiiiiiiiii} can be applied to simplify the computation of the tangent complex of higher Lie groupoids. Exemplarily, we consider the case of \textbf{String Lie 2-group}.

\begin{example}
    Given a compact, connected and simple Lie group G, the String Lie 2-group $\text{String}(G)$ is defined as a central Lie 2-group extension of $G$ by $B\mathbb{S}^1$ as specified in \cite{schommer}, where the extension class is the generator of $H^3(BG,\mathbb{S}^1)\cong \mathbb{Z}$. Let $(U^{(n)}_i)_{i\in I^n}$ be a \s cover of $BG$ as defined in \cite{WZ}. Then the first three levels of $\text{String}(G)$ (seen as a \s manifold $X_{\bu}$) are defined as
    \begin{gather*}
        X_0=\ast,\:X_1=U_{[0]}^{(1)}, 
        X_2=(((U_{[1]}^{(1)})^{\times 3}\times_{(U_{[0]}^{(1)})^{\times 3}} U_{[0]}^{(2)})/U_{[1]}^{(2)})\times \mathbb{S}^1
    \end{gather*}
    where $U^{(p)}_{[q]}$ denotes the disjoint union $\bigsqcup_{i_0,\dots,i_q\in I^p}U_{i_0}\cap \dots\cap U_{i_q}$ of $(q+1)$-fold intersections of the open cover $(U^{(p)}_i)_{i\in I^p}$ of $(BG)_p$ (see \cite{WZ}, \cite{schommer} for details). The 2-face maps $d_i^2:X_2\rightarrow X_1$ are given by 
    \begin{gather*}
        d_0(([x^{(1)}_{i_0,i_1},x^{(1)}_{j_0,j_1}, x^{(1)}_{k_0,k_1}, x^{(2)}_{l_0}],s))=x^{(1)}_{j_0},\:\:\:\:
         d_1(([x^{(1)}_{i_0,i_1},x^{(1)}_{j_0,j_1}, x^{(1)}_{k_0,k_1}, x^{(2)}_{l_0}],s))=x^{(1)}_{k_0}\\
         d_2(([x^{(1)}_{i_0,i_1},x^{(1)}_{j_0,j_1}, x^{(1)}_{k_0,k_1}, x^{(2)}_{l_0}],s))=x^{(1)}_{i_0}.
    \end{gather*}

   \noindent By choosing suitable small open neighborhoods $V_1$ of $X_0\subseteq X_1$, $V_2$ of $X_0\subseteq  X_2$ as in the proof of Theorem \ref{yummiiiiiiiiii} such that $((U_{[1]}^{(1)})^{\times 3}\times_{(U_{[0]}^{(1)})^{\times 3}} U_{[0]}^{(2)})/U_{[1]}^{(2)}$ covers $V_1\times V_1$, it follows that ${\text{ker}Tp^1_0}_{|\ast}=\mathfrak{g}$ and ${({\text{ker}Tp^2_0}_{|V_2})}_{|\ast}=\mathbb{R}$, where $\mathfrak{g}$ denotes the Lie algebra of $G$. Thus, the tangent complex of $\text{String}(G)$ is given by \begin{gather*}
       \dots\rightarrow 0\rightarrow \mathbb{R}\rightarrow \mathfrak{g}\rightarrow 0.
   \end{gather*}
\end{example}

\appendix 

\section{Simplicial vector spaces}

\noindent We will now provide a proof of Lemma \ref{yone}, showing that simplicial vector spaces satisfy the undermentioned generalized Kan conditions. This fact was used in the proof of Theorem \ref{yummiiiiiiiiii}.
\begin{lemm}
\label{proofapp1}
Let $V_{\bu}$ be a \s vector space. For $n,m\in \mathbb{N},\:j\in \mathbb{N}_0$, $0\leq j\leq m<n$, the generalized horn projections  $p^n_{\{j,m+1,\dots,n\}}:V_n\rightarrow \bigwedge^n_{\{j,m+1,\dots,n\}}(V)\subseteq V_{n-1}^{\oplus ([m]\setminus \{j\})}$ are surjective.
\end{lemm}
\begin{proof}
We show that there exists an explicit right inverse for $p^n_{\{j,m+1,\dots,n\}}$.
Given $(v_{n-1}^0,\dots,\widehat{v_{n-1}^j},\dots,v_{n-1}^m)\in\bigwedge^n_{\{j,m+1,\dots,n\}}(V)$, the element
    \begin{gather*}
   v_n:= \sum_{l=0}^{m-1}(-1)^l\sum_{\substack{0\leq j_0<\dots<j_l\leq m\\   j_0,\dots,j_l\neq j}}s_{j_l}d_{j_l}\dots s_{j_1}d_{j_1}s_{j_0}(v_{n-1}^{j_0})
\end{gather*}
turns out to be a horn filler, i.e. $d_i(v_n)=v_{n-1}^i$ for $i\in \{0,\dots m\}\setminus \{j\}$. Note that for $i\in \{0,\dots m\}\setminus \{j\}$
\begin{gather*}
    d_i(v_n)=\underbracket{d_is_i(v_{n-1}^i)}_{=v_{n-1}^i}+d_i(\sum_{j_0\notin \{i,j\}}s_{j_0}(v_{n-1}^{j_0})-\sum_{\substack{j_0<j_1\\   j_0,j_1\neq j}}s_{j_1}d_{j_1}s_{j_0}(v_{n-1}^{j_0})+\sum_{l=2}^{m-1}(-1)^l\sum_{\substack{0\leq j_0<\dots<j_l\leq m\\   j_0,\dots,j_l\neq j}}s_{j_l}d_{j_l}\dots s_{j_1}d_{j_1}s_{j_0}(v_{n-1}^{j_0})),
\end{gather*}
so it needs to be shown that the term 
\begin{gather*}
    d_i(\sum_{j_0\notin \{i,j\}}s_{j_0}(v_{n-1}^{j_0})-\sum_{j_0<j_1;j_0,j_1\neq j}s_{j_1}d_{j_1}s_{j_0}(v_{n-1}^{j_0})+\dots\pm s_{m}d_{m}\dots \widehat{s_jd_j}\dots s_0(v_{n-1}^0))
\end{gather*}
vanishes. As $v_n$ is defined as an alternating sum of sums, it is enough to verify that for $l\in \{0,\dots,m-2\}$ the term $d_i(\sum_{j_0<\dots<j_l;j_0,\dots,j_l\notin\{j,i\}}s_{j_l}d_{j_l}\dots s_{j_0}(v_{n-1}^{j_0}))$ equals $d_i(\sum_{j_0<\dots<j_{l+1};i\in\{j_0,\dots,j_{l+1}\};j_0,\dots,j_{l+1}\neq j}s_{j_{l+1}}d_{j_{l+1}}\dots s_{j_0}(v_{n-1}^{j_0}))$.\\
\\
Using simplicial identities, this follows from 
\begin{gather*}
    d_i(\sum_{\substack{i<j_0<\dots<j_l \\ j_0,\dots,j_l\neq j}}s_{j_l}d_{j_l}\dots s_{j_0}(v_{n-1}^{j_0}))=\sum_{\substack{i<j_0<\dots<j_l \\ j_0,\dots,j_l\neq j}}s_{j_l-1}d_{j_l-1}\dots s_{j_0-1}d_i(v_{n-1}^{j_0})\\
    =\sum_{\substack{i<j_0<\dots<j_l \\ j_0,\dots,j_l\neq j}}s_{j_l-1}d_{j_l-1}\dots s_{j_0-1}d_{j_0-1}(v_{n-1}^{i})=\sum_{\substack{i=j_0<\dots<j_{l+1} \\ j_0,\dots,j_{l+1}\neq j}}s_{j_{l+1}-1}d_{j_{l+1}-1}\dots s_{j_1-1}d_{j_1-1}d_{j_0}s_{j_0}(v_{n-1}^{j_0})\\
    =d_i(\sum_{\substack{i=j_0<\dots<j_{l+1} \\ j_0,\dots,j_{l+1}\neq j}}s_{j_{l+1}}d_{j_{l+1}}\dots s_{j_1}d_{j_1}s_{j_0}(v_{n-1}^{j_0})),\\
     d_i(\sum_{\substack{j_0<\dots<j_q<i<j_{q+1}<\dots<j_l \\ j_0,\dots,j_l\neq j}}s_{j_l}d_{j_l}\dots s_{j_0}(v_{n-1}^{j_0}))\\
     =\sum_{\substack{j_0<\dots<j_q<i<j_{q+1}<\dots<j_l \\ j_0,\dots,j_l\neq j}}s_{j_l-1}d_{j_l-1}\dots s_{j_{q+1}-1}d_{j_{q+1}-1}d_is_{j_q}d_{j_q}\dots s_{j_0}(v_{n-1}^{j_0})\\
     =\sum_{\substack{j_0<\dots<j_q<i=j_{q+1}<\dots<j_{l+1} \\ j_0,\dots,j_{l+1}\neq j}}s_{j_{l+1}-1}d_{j_{l+1}-1}\dots s_{j_{q+2}-1}d_{j_{q+2}-1}d_is_{j_{q+1}}d_{j_{q+1}}s_{j_q}d_{j_q}\dots s_{j_0}(v_{n-1}^{j_0})
     \end{gather*}
     \begin{gather*}
     =d_i(\sum_{\substack{j_0<\dots<j_q<i=j_{q+1}<\dots<j_{l+1} \\ j_0,\dots,j_{l+1}\neq j}}s_{j_{l+1}}d_{j_{l+1}}\dots s_{j_0}(v_{n-1}^{j_0}))\:\:\:\text{for}\:\:q\in\{0,\dots,l-1\}\:\:\text{and}\\
         d_i(\sum_{\substack{j_0<\dots<j_l<i\\ j_0,\dots,j_l\neq j}}s_{j_l}d_{j_l}\dots s_{j_0}(v_{n-1}^{j_0}))=  d_i(\sum_{\substack{j_0<\dots<j_l<j_{l+1}=i\\ j_0,\dots,j_{l+1}\neq j}}s_{j_{l+1}}d_{j_{l+1}}\dots s_{j_0}(v_{n-1}^{j_0})).
\end{gather*}
\end{proof}

\noindent In the following it will be demonstrated that besides the normalization functor $N:\catname{sVect}\rightarrow \catname{Ch_{\geq 0}(Vect)}$ defined in section 3.2, there exists another normalization functor $\tilde{N}:\catname{sVect}\rightarrow \catname{Ch_{\geq 0}(Vect)}$ which is naturally isomorphic to $N$. When extending this isomorphism, one can show that the two versions of the tangent complex presented in Remark \ref{twoversions} are naturally isomorphic.

\begin{defi}
   Let $V_{\bu}$ be a \s vector space. For $n\in \mathbb{N},\:m,i\in \mathbb{N}_0$, $0\leq i \leq m\leq n$, define the following linear subspaces of $V_n$:
   \begin{gather*}
       N^n_{i,m}(V):=\begin{cases}V_n\:\:\:\:\:\:\:\:\:\:\:\:\:\:\:\:\:\:\:\:\:\:\:\:\:\:\:\:\:\:\:\:\:\:\:\:\:\:\:\:\:\:\:\:\:\:\:\:\:\:\:\:\:\:\:\:\:\:\:\:\:\:\:\:\:\:\:\:\text{if}\:\:m=0\\
       \bigcap_{j\in\{0,\dots,m\}\setminus\{i\}}\,\text{ker}\,(d_j:V_n\rightarrow V_{n-1})\:\:\:\;\text{if}\:\:m>0\\
       \end{cases}.
   \end{gather*}
   Further, let $(N(V)_{\bu},\partial_{\bu})$, $(\tilde{N}(V)_{\bu},\tilde{\partial}_{\bu})$ be the connected chain complexes of vector spaces which are given by
   \begin{gather*}
       N(V)_n=N_{n,n}^n(V),\:n\in \mathbb{N}_0\:\:\:\:\:\:\:\:\:\partial_n=(-1)^nd_n^n,\:n\in \mathbb{N}\\
       \tilde{N}(V)_n=N_{0,n}^n(V),\:n\in \mathbb{N}_0\:\:\:\:\:\:\:\:\:\tilde{\partial}_n=d_0^n,\:n\in \mathbb{N}.
   \end{gather*}
\end{defi}

\begin{lemm}
\label{useless}
Let $V_{\bu}$ be a \s vector space. There exists a natural chain isomorphism between the associated chain complexes of vector spaces $(N(V)_{\bu},\partial_{\bu})$ and $(\tilde{N}(V)_{\bu},\tilde{\partial}_{\bu})$.
\end{lemm}
\begin{proof}
For $n\in \mathbb{N},\:m\in \mathbb{N}_0,\:0\leq m\leq n$, define the following projections:
\begin{gather*}
    \gamma^n_{0,m}:=(id-s_0d_1)\circ\dots \circ (id-s_{m-1}d_m):V_n\rightarrow N^n_{0,m}(V)\\
    \gamma^n_{m,m}:=(id-s_{m-1}d_{m-1})\circ \dots \circ (id-s_0d_0):V_n\rightarrow N^n_{m,m}(V)
\end{gather*}
 and set $\gamma^n_{0,0}:=id_{V_n}$. The restrictions ${\gamma^n_{m,m}}_{|N^n_{0,m}},\:{\gamma^n_{0,m}}_{|N^n_{m,m}}$, $m\in \{0,\dots,n\}$, prove to be inverses of each other. For a fix $n$ we show this inductively over $m$. For $m=1$ one has 
 \begin{gather*}
     \gamma^n_{1,1}\circ {\gamma^n_{0,1}}_{|N^n_{1,1}}={(id-s_0d_0)}_{|N^n_{1,1}}=id_{N^n_{1,1}},\\
     \gamma^n_{0,1}\circ {\gamma^n_{1,1}}_{|N^n_{0,1}}={(id-s_0d_1)}_{|N^n_{0,1}}=id_{N^n_{0,1}}.
 \end{gather*}
 Assuming that the statement is true until $m$, it follows with the \s identities that 
 \begin{gather*}
     \gamma^n_{m+1,m+1}\circ {\gamma^n_{0,m+1}}_{|N^n_{m+1,m+1}}=(\gamma^n_{m,m}-s_md_m\gamma^n_{m,m})\circ{(\gamma^n_{0,m}-\gamma^n_{0,m}s_md_{m+1})}_{|N^n_{m+1,m+1}}\\
     =(\gamma^n_{m,m}\circ {\gamma^n_{0,m}}-\gamma^n_{m,m}\circ \gamma^n_{0,m}s_md_{m+1}-s_md_m\gamma^n_{m,m}\circ {\gamma^n_{0,m}}+s_md_m\gamma^n_{m,m}\circ {\gamma^n_{0,m}}s_md_{m+1})_{|N^n_{m+1,m+1}}\\
     =(id-s_md_{m+1}-s_md_m+s_md_{m+1})_{|N^n_{m+1,m+1}}=id_{N^n_{m+1,m+1}},\\
     \gamma^n_{0,m+1}\circ {\gamma^n_{m+1,m+1}}_{|N^n_{0,m+1}}=(\gamma^n_{0,m}-\gamma^n_{0,m}s_md_{m+1})\circ(\gamma^n_{m,m}-s_md_m\gamma^n_{m,m})_{|N^n_{0,m+1}}\\
     =( \gamma^n_{0,m}\gamma^n_{m,m}-\gamma^n_{0,m}s_md_m\gamma^n_{m,m}-\gamma^n_{0,m}s_md_{m+1}\gamma^n_{m,m}+\gamma^n_{0,m}s_md_m\gamma^n_{m,m})_{|N^n_{0,m+1}}\\
     =(id-\gamma^n_{0,m}s_m\gamma^n_{m,m}d_{m+1})_{|N^n_{0,m+1}}=id_{N^n_{0,m+1}}.
 \end{gather*}
Moreover, using \s identities it is easy to verify that $(-1)^nd_n\circ \gamma^n_{n,n}=d_0=\gamma^{n-1}_{n-1,n-1}\circ d_0$ for $n\in \mathbb{N}$. Thus, the maps $\gamma^n_{n,n}:N^n_{0,n}(V)\rightarrow N^n_{n,n}(V),\:n\in \mathbb{N}_0$, commute with the differentials and form a chain isomorphism.
\end{proof}

\section{Background from \cite{ZR}}
\noindent For the sake of better readability we recall several results and definitions from \cite{ZR} which are most relevant for the purpose of this paper.
\addtocontents{toc}{\protect\setcounter{tocdepth}{1}}
\subsection{Tangent bundles of graded manifolds} (see \cite[sec.\,2.2, 2.3]{ZR})
\begin{thm}
\label{rep}
Let $E_{\bu}\rightarrow M$ be a graded vector bundle. Then the presheaf $\text{Hom}(D,\mathfrak{S}(E_{\bu}))$ is representable and a connection on the dual $E_{\bu}^{\ast}$ induces a (non-canonical) natural isomorphism $Hom(D,\mathfrak{S}(E_{\bu}))\cong \text{Yon}(\mathfrak{S}(E_{\bu}\oplus E_{\bu}[1]\oplus TM[1]))$.
\end{thm}
\noindent In the case that $E_{\bu}$ is the trivial graded vector bundle over $M$, i.e. the associated graded manifold $\mathfrak{S}(E_{\bu})$ is just an ordinary manifold $M$, the above isomorphism $\text{Hom}(D,M)\cong \text{Yon}(\mathfrak{S}(TM[1]))$ does not depend on the choice of a connection and is natural in $M$. 
\begin{lemm}
\label{simpleD}
For a smooth manifold $M$ the presheaf $\text{Hom(D,M)}$ is canonically represented by $\mathfrak{S}(TM[1])$. Given a smooth map $\phi:M\rightarrow N$ between smooth manifolds, the pushforward 
\begin{gather*}
    \text{Yon}(\mathfrak{S}(TM[1]))\cong \text{Hom}(D,M)\xrightarrow[\phi_{\ast}]{}\text{Hom}(D,N)\cong \text{Yon}(\mathfrak{S}(TN[1]))
\end{gather*}
 corresponds to the map $\mathfrak{S}(T\phi[1]):\mathfrak{S}(TM[1])\rightarrow \mathfrak{S}(TN[1])$.
\end{lemm}
\noindent Considering the next level $D_{[1]}=D^2$ of $D_{\bu}$, a choice of connection on $T^{\ast}M$ induces a natural isomorphism $\text{Hom}(D^2,M)\cong \text{Hom}(D,\mathfrak{S}(TM[1]))\cong \text{Yon}(\mathfrak{S}(TM[1]\oplus TM[1]\oplus TM[2]))$. In general, the presheaves $\text{Hom}(D^k,M),\:k\in \mathbb{N}$, called the \textbf{$\mathbf{k}$th iterated tangent bundles of $\mathbf{M}$}, are representable which follows inductively from Theorem \ref{rep}.
\begin{coro}
\label{rep2}
Let $M$ be a smooth manifold. For $k\in \mathbb{N}_{\geq 2}$ the presheaf $\text{Hom}(D^k,M)$ is representable and a connection on $T^{\ast}M$ induces a (non-canonical) isomorphism $\text{Hom}(D^k,M)\cong \text{Yon}(\mathfrak{S}(T[1]^kM))$ 
where $T[1]^kM$ is the graded vector bundle $T[1]^kM:=\bigoplus_{i=1}^k(TM[i])^{\oplus \binom{k}{i}}$.
\end{coro}
\subsection{Compatible connections on (graded) vector bundles} (see \cite[sec.\,2.5]{ZR})\par
\vspace{1.7mm}
\noindent To formulate a functorial version of Theorem \ref{rep}, the concept of compatible connections on (graded) vector bundles is introduced.
\begin{defi}
   Let $E\rightarrow M$, $\tilde{E}\rightarrow N$ be smooth vector bundles, $\phi:M\rightarrow N$ be a smooth map and $\phi^{\ast}\tilde{E}$ be the pullback of $\tilde{E}$ along $\phi$. A \textbf{vector bundle comorphism} $E \dashleftarrow \tilde{E}$ over $\phi$ is a vector bundle morphism $\phi^{\ast}\tilde{E}\rightarrow E$ over $M$.
\end{defi}
\begin{defi}
Let $F:E\rightarrow \tilde{E}$ be a vector bundle morphism between vector bundles $E,\tilde{E}$ over the same base $M$ and let $J^1(E), J^1(\tilde{E})$ be the first jet bundles of $E,\tilde{E}$, respectively. Define $J^1F:J^1(E)\rightarrow J^1(\tilde{E})$ to be the vector bundle morphism over $M$ mapping a jet $J^1_ps\in J^1E_p,\:p\in M$, to the jet $J^1_pFs\in J^1\tilde{E}_p$.
\end{defi}
\begin{defi}
   Let $\psi:E\rightarrow J^1(E),\:\tilde{\psi}:\tilde{E}\rightarrow J^1(\tilde{E})$ be \textbf{linear  connections} on smooth vector bundles $E\rightarrow M$, $\tilde{E}\rightarrow N$ and let $\varphi:E\dashleftarrow \tilde{E}$ be a vector bundle comorphism over $\phi:M\rightarrow N$. Then $\psi,\tilde{\psi}$ are called \textbf{compatible} w.r.t. $\varphi$ if the diagram
   \begin{center}
\begin{tikzcd}
J^1(E)                                            & E \arrow[l, "\psi"']                                                            \\
J^1(\phi^{\ast}\tilde{E}) \arrow[u, "J^1\varphi"] & \phi^{\ast}\tilde{E} \arrow[u, "\varphi"'] \arrow[l, "\phi^{\ast}(\tilde{\psi})"]
\end{tikzcd}
   \end{center}
   commutes, where 
   $\phi^{\ast}(\tilde{\psi}):\phi^{\ast}\tilde{E}\rightarrow J^1(\phi^{\ast}\tilde{E})$ denotes the pullback connection of $\tilde{\psi}$.
   Two connections $\psi$, $\tilde{\psi}$ defined on the duals of vector bundles $E\rightarrow M$, $\tilde{E}\rightarrow N$ are called \textbf{compatible w.r.t. a vector bundle morphism} $F:E\rightarrow \tilde{E}$ over $\phi:M\rightarrow N$ if they are compatible w.r.t. the natural vector bundle comorphism $E^{\ast}\dashleftarrow \tilde{E}^{\ast}$ corresponding to $F$.
   \begin{rema}
   \label{remaconn}
   \begin{enumerate}
    \renewcommand{\labelenumi}{\roman{enumi})}
       \item Let $E\rightarrow M,F\rightarrow N$ be vector bundles having compatible connections w.r.t. a vector bundle comorphism $\varphi:E\dashleftarrow F$ over $\phi:M\rightarrow N$. Let $U,V$ be open subsets of $M$, $N$ such that $\phi(U)\subseteq V$. Then the restricted connections on $E_{|U},\:F_{|V}$ are compatible w.r.t the restriction $\varphi:E_{|U}\dashleftarrow F_{|V}$.
       \item Let $E\rightarrow M$ be a vector bundle, $U\subseteq M$ be an open subset and $\eta$ be a connection on $E^{\ast}$. A connection $\psi$ on $E^{\ast}_{|U}$ is compatible to $\eta$ w.r.t. the embedding $E_{|U}\hookrightarrow E$ if and only if $\eta_{|U}=\psi$.
       \item Given a submersion $p:M\rightarrow N$ between smooth manifolds and a connection $\psi$ on $T^{\ast}N$, there exists a connection on $T^{\ast}M$ which is compatible to $\psi$ w.r.t. $Tp:TM\rightarrow TN$ (see \cite[Cor.\,2.33]{ZR}).
       \item By considering level-wise connections the above definitions and statements can be generalized for graded vector bundles in a straightforward manner. Connections on graded vector bundles are meant to be level-wise connections in the following.
   \end{enumerate}
   \end{rema}
\end{defi}
\begin{thm}
  Let $F_{\bu}:E_{\bu}\rightarrow \tilde{E}_{\bu}$ be a morphism between graded vector bundles covering the smooth map $\phi:M\rightarrow N$. Given compatible connections on $E_{\bu}^{\ast},\:\tilde{E}_{\bu}^{\ast}$ w.r.t. $F_{\bu}$, the pushforward $\mathfrak{S}(F_{\bu})_{\ast}:Hom(D,\mathfrak{S}(E_{\bu}))\rightarrow Hom(D,\mathfrak{S}(\tilde{E}_{\bu}))$ corresponds to the graded vector bundle morphism $F_{\bu}\oplus F_{\bu}[1]\oplus T\phi[1]:E_{\bu}\oplus E_{\bu}[1]\oplus TM[1]\rightarrow \tilde{E}_{\bu}\oplus \tilde{E}_{\bu}[1]\oplus TN[1]$. That is, the following diagram commutes
  \begin{center}
\begin{tikzcd}
{Hom(D,\mathfrak{S}(E_{\bullet}))} \arrow[d, "\cong"] \arrow[rrr, "(\mathfrak{S}(F_{\bullet}))_{\ast}"]                                                    &  &  & {Hom(D,\mathfrak{S}(\tilde{E}_{\bullet}))} \arrow[d, "\cong"]                     \\
{Yon(\mathfrak{S}(E_{\bullet}\oplus E_{\bullet}[1]\oplus TM[1]))} \arrow[rrr, "{(\mathfrak{S}(F_{\bullet}\oplus F_{\bullet}[1]\oplus T\phi[1]))_{\ast}}"'] &  &  & {Yon(\mathfrak{S}(\tilde{E}_{\bullet}\oplus\tilde{E}_{\bullet}[1]\oplus TN[1])) }
\end{tikzcd}
  \end{center}
  where the vertical arrows are the isomorphisms induced by the connections on $E_{\bu}^{\ast},\:\tilde{E}_{\bu}^{\ast}$ according to Theorem\,\ref{rep}.
\end{thm}
\begin{coro}
\label{moreD}
Let $p:M\rightarrow N$ be a submersion between smooth manifolds and suppose that there exist compatible connections on $T^{\ast}M,\:T^{\ast}N$ w.r.t. the tangent map $Tp:TM\rightarrow TN$. Then, applying Corollary \ref{rep2}, the pushforward $p_{\ast}:Hom(D^k,M)\rightarrow Hom(D^k,N),\:k\in \mathbb{N}_{\geq 2}$, corresponds to the graded vector bundle morphism $T[1]^kp:=\bigoplus_{i=1}^k(Tp[i])^{\oplus \binom{k}{i}}:T[1]^kM\rightarrow T[1]^kN$.
\end{coro}
\subsection{}
 \noindent As stated in the proof of Lemma \ref{weakercase}, under the weaker the assumption that horn projections are submersions (instead of surjective submersions) the inductive proof of Theorem \ref{yoooo} in \cite{ZR} can be straightforwardly extended, with the exception of Lemma 3.16 in \cite{ZR}. Assuming that horn projections are submersions, Lemma 3.16 is proven below.
\begin{lemm}
\label{proofanhang}
Let $X_{\bu}$ be a local Lie $\infty$-groupoid with given compatible connections on $T^{\ast}X_k$, $T^{\ast}\bigwedge^k_0(X)$ for $k\geq 2$. Further, let $T[1]^kp^k_0:T[1]^kX_k\rightarrow T[1]^k\bigwedge^k_0(X),\:k\in \mathbb{N}$, be the morphisms between graded manifolds which correspond to $(p^k_0)_{\ast}:\text{Hom}(D^k,X_k)\rightarrow \text{Hom}(D^k,\bigwedge^k_0(X))$ according to Lemma \ref{simpleD} and Corollary \ref{moreD}, and let $\widetilde{T[1]^kp^k_0},\:k\in \mathbb{N}$, be the respective induced maps $T[1]^kX_k/TX_k[k]\rightarrow T[1]^k\bigwedge^k_0(X)/T\bigwedge^k_0(X)[k]$. Lastly, let $q_k$, $k\in \mathbb{N}$, denote the projections $T[1]^k\bigwedge^k_0(X)\rightarrow T[1]^k\bigwedge^k_0(X)/$ $T\bigwedge^k_0(X)[k]$. Then for $k\in \mathbb{N}$, the fiber product associated to the diagram of graded manifolds
\begin{center}
\begin{tikzcd}
                                                             &  & {T[1]^k\bigwedge^k_0(X)} \arrow[d, "q_k"]     \\
{T[1]^kX_k/TX_k[k]} \arrow[rr, "{\widetilde{T[1]^kp^k_0}}"'] &  & {T[1]^k\bigwedge^k_0(X)/T\bigwedge^k_0(X)[k]}
\end{tikzcd}
\end{center}
exists and is given by $T[1]^kX_k/\text{ker}\,Tp^k_0[k]$.
\end{lemm}
\begin{proof}
Like in the proof of Lemma \ref{weakercase}, it suffices to compute the fiber product in the category of $\mathbb{Z}_{<0}$-graded vector bundles. We first notice that $T[1]^kp^k_0:T[1]^kX_k/\text{ker}\,Tp^k_0[k]\rightarrow T[1]^k\bigwedge^k_0(X)$ and the map $T[1]^kX_k/\text{ker}\,Tp^k_0[k]\rightarrow T[1]^kX_k/TX_k[k]$ make the above diagram commute. When checking that $T[1]^kX_k/\text{ker}\,Tp^k_0[k]$ is the fibre product, it is sufficient to do so in each 
degree.\\
Denoting $T^k_{-i}(X_k):=TX_k^{\oplus \binom{k}{i}},\:T^k_{-i}(\bigwedge^k_0(X)):=T\bigwedge^k_0(X)^{\oplus \binom{k}{i}}$ to be the $(-i)$-components of $T[1]^kX_k/\text{ker}\,Tp^k_0[k]$ for $1\leq i\leq k-1$, one has the following
pullback diagrams
\begin{center}
\begin{tikzcd}
T^k_{-i}(X_k) \arrow[rr, "(Tp^k_0)^{\oplus\binom{k}{i}}"] \arrow[d, "id"'] &  & T^k_{-i}(\bigwedge^k_0(X)) \arrow[d, "id"] \\
T^k_{-i}(X_k) \arrow[rr, "(Tp^k_0)^{\oplus\binom{k}{i}}"']                 &  & T^k_{-i}(\bigwedge^k_0(X))                
\end{tikzcd}.
\end{center}
\noindent In degree $-k$ the pullback diagram is given by 
\begin{center}
\begin{tikzcd}
(p^k_0)^{\ast}(T\bigwedge^k_0(X)) \arrow[d, "{(0,id_{X_k})}"'] \arrow[r] & T\bigwedge^k_0(X) \arrow[d, "{(0,id_{\bigwedge^k_0(X)})}"] \\
0 \arrow[r, "{(0,p^k_0)}"']                                              & 0                                                         
\end{tikzcd}
\end{center}
where the pullback bundle $(p^k_0)^{\ast}(T\bigwedge^k_0(X))$ is isomorphic to $TX_k/\text{ker}\,Tp^k_0$ since the horn projection $p^k_0$ is a submersion.
\end{proof}

\addtocontents{toc}{\protect\setcounter{tocdepth}{2}}



\begin{thebibliography}{xxxxxxxxxxxxxxxxxxx}
  
  



 \bibitem[Du02]{duskin}  Duskin, J.W.: Simplicial matrices and the nerves of weak n-categories. I: Nerves of
bicategories. Theory Appl. Categ. (electronic), vol.\,9, pp.\,198–308 (2002)

\bibitem[BP13]{Nman}   Bonavolontà, G.; Poncin, N.: On the category of Lie n-algebroids. J. Geom. Phys., 73:70–90 (2013)

\bibitem[Ge14]{getzler} Getzler, E.: Differential forms on stacks. Slides from a minicourse at Winter School in Mathematical Physics, Les Diablerets (2014)
  
 \bibitem[He08]{henri} Henriques, A.: Integrating L-infinity algebras. Compositio Mathematica, vol.\,144, pp.\,1017 - 1045 (2008)
 
 \bibitem[Hu04]{hu}  Huybrechts, D.: Complex geometry - an introduction. Springer (2004)
  
 \bibitem[KS21]{salni} Kotov, A.; Salnikov, V.: The category of $\mathbb{Z}$-graded manifolds: What happens if you do not stay positive. arXiv:\,2108.13496 [math.DG] (2021)
 
     \bibitem[Le13]{lee}
    Lee, J.M.: Introduction to Smooth Manifolds. Spinger, New York (2013)

 \bibitem[LRWZ23]{ZR}  Li, D.; Ryvkin, L.; Wessel, A.;  Zhu, C.: Differentiating $L_{\infty}$ groupoids \mbox{Part\,I}.  	arXiv:2309.00901\,[math.DG] (2023)   

  \bibitem[Ma92]{may}  May, P.: Simplicial Objects in Algebraic Topology. University of Chicago Press (1967)
  
  \bibitem[Ri14]{riehl} Riehl, E.: Weighted limits and colimits. in: Categorical Homotopy Theory (New Mathematical Monographs, pp. 99-120), Cambridge University Press (2014)
  
    \bibitem[Se06]{severa} Severa, P.: L-\,infinity algebras as 1-jets of simplicial manifolds (and a bit beyond). arXiv:math/0612349 [math.DG] (2006)

    \bibitem[SP11]{schommer} Schommer-Pries, C.: Central extensions of smooth 2-groups and a finite-dimensional string 2-group. Geom. Topol., 15(2), pp.609–676 (2011)
    
\bibitem[Vy21]{vysoki}   Vysoky, J.: On the global theory of Graded Manifolds. arXiv:\,2105.02534[math.DG] (2021)

\bibitem[WZ16]{WZ}  Wockel, C.; Zhu, C.: Integrating central extensions of Lie algebras via Lie 2-groups. J. Eur. Math. Soc. (JEMS), 18(6), pp. 1273–1320 (2016)
    
 
 \bibitem[Zh09]{chenchang}  Zhu, C.: n-groupoids and stacky groupoids. International Mathematics Research Notices 2009, pp.\,4087-4141 (2009)
 

 








    
    
    



    
   

    

                               
\end{thebibliography}
\end{document}